\documentclass[a4paper,11pt,reqno]{amsart}

\usepackage[utf8]{inputenc}
\usepackage[english]{babel}
\usepackage{amsmath}
\usepackage{amsfonts}
\usepackage{amsthm}
\usepackage{mathtools}
\usepackage{float}
\usepackage{stmaryrd}
\usepackage{graphics}
\usepackage{graphicx}
\usepackage{subfig}
\usepackage{enumitem}
\usepackage{datetime}
\usepackage{bm}
\usepackage[colorlinks=false]{hyperref}
\usepackage{color}

\newcommand{\mc}{\mathcal}

\newcommand{\eps}{\varepsilon}
\renewcommand{\d}{\,\mathrm{d}}
\renewcommand{\div}{\mathrm{div}}
\newcommand{\PP}{\mathcal{P}}
\renewcommand{\AA}{\mathcal{A}}
\newcommand{\CC}{\mathcal{C}}

\def\R{\mathbb{R}}
\renewcommand{\to}{\rightarrow}

\numberwithin{equation}{section}
\newtheorem{thm}{Theorem}[section]
\newtheorem{defi}[thm]{Definition}
\newtheorem{prop}[thm]{Proposition}
\newtheorem{lemma}[thm]{Lemma}
\newtheorem{cor}[thm]{Corollary}

\theoremstyle{definition}
\begingroup
\newtheorem{rmk}[thm]{Remark}

\endgroup

\theoremstyle{remark}

\begin{document}
	\author[F. Riva and E. Rocca]{Filippo Riva and Elisabetta Rocca}
	
	\title[Sharp interface limit for tumor growth models]{A rigorous approach to the sharp interface limit for phase-field models of tumor growth}
	
	\begin{abstract}
		In this paper we consider two diffuse interface models for tumor growth coupling a Cahn-Hilliard type equation for the tumor phase parameter to a reaction-diffusion type equation for the nutrient. The models are distinguished by the presence of two different coupling source terms. For such problems, we address the question of the limit, as the diffuse interface parameter tends to zero, from diffuse interface models to sharp interface ones, justifying rigorously what was deduced via formal asymptotics in \cite{GarLamSitSty}. The resulting evolutions turn out to be varifold solutions to Mullins-Sekerka type flows for the tumor region suitably coupled with the equation for the nutrient. 
	\end{abstract}
	
	\maketitle
	
	{\small
		
		\par
	}

	\pagenumbering{arabic}
	
	\medskip
	
	\tableofcontents
	
	\section*{Introduction}
	
	The problem of modelling tumor growth dynamics has recently become a major issue in applied mathematics (see~\cite{CL10,WLFC08}). 
	The morphological evolution of a growing solid tumor is the result of the dynamics of a complex system that includes many nonlinearly interacting factors, including cell-cell and cell-matrix adhesion, mechanical stress, cell motility and angiogenesis just to name a few. Numerous mathematical models have been developed to study various aspects of tumor progression and this has been an area of intense research interest (see the recent reviews \cite{BLM08,CEtAl08, FBG06, FBM07, GP07, LEtAl10}). 
	
	The existing models can be divided into two main categories: continuum models and discrete models. We concentrate on the former ones. There the necessity of dealing with multiple interacting constituents has led to the consideration of diffuse-interface models based on continuum mixture theory (see, for instance, \cite{CLW09} and references therein). 
	In the diffuse approach, sharp interfaces between different species are replaced by narrow transition layers, of small thickness $\varepsilon$, that arise due to differential adhesive
	forces among the cell-species. The main advantages of the diffuse interface formulation are: 
	- it eliminates the need to enforce complicated boundary conditions across the tumor/host tissue and other species/species interfaces that would have to be satisfied if the interfaces were assumed sharp, and
	- it eliminates the need to explicitly track the position of interfaces, as is required in the sharp interface framework. In this approach the tumor concentration $\varphi$ (representing the local proportion of the tumor phase) varies in the interval $[-1,1]$, with the convention that $\varphi\equiv 1$ means that we are in the tumor phase, $\varphi\equiv -1$ means that we are in the healthy phase, while $\varphi\in (-1,1)$ in the diffuse interface between the pure phases. 
	
	Such models generally consist of Cahn-Hilliard equations with transport and reaction terms which govern the tumour concentration parameter $\varphi$. The reaction terms depend on the nutrient concentration $\sigma$ (e.g., oxygen) which obeys an advection-reaction-diffusion equation. In the simplest case of a two-phase model (where we only consider one tumor phase and the healthy phase) and neglecting velocities, the resulting PDE system turns out to be of this type 
	\begin{equation}\label{diffuse}
		\begin{aligned}
			&\dot\varphi-\Delta\mu=R(\varphi,\sigma,\mu), \quad \mu =-\varepsilon \Delta\varphi+\frac 1\varepsilon F'(\varphi),\\		
			&\dot\sigma-\Delta\sigma=S(\varphi,\sigma,\mu),
		\end{aligned}
	\end{equation}
	where $F$ denotes the typical double-well potential associated with the Ginzburg-Landau free-energy functional: 
	\begin{equation}\label{double}
		F(u)=\frac 14(1-u^2)^2.
	\end{equation}
	In \eqref{diffuse}, $R$ and $S$ are specific source terms governing the proliferation and the death of tumor cells, as well as the consumption of the nutrient by the tumour. Different choices for $R$ and $S$ are possible. 
	In \cite{FriGrasRoc, HZO11} linear phenomenological laws for chemical reactions, like
	\begin{align}\label{lin}
		R = -S = P(\varphi)(\sigma - \mu),
	\end{align}
	are considered. Here, $P$ can be regarded as a proliferation function; for instance, it may have the
	form ${P}(u)=P_0(1-u^2)\chi_{[-1,1]}(u)$ for $u\in\mathbb{R}$, $P_0>0$.
	Otherwise, in  \cite{CWSL14} and \cite{GarLamSitSty} they consider the following choice: 
	\begin{equation}\label{phe}
		R = (\PP \sigma - \AA)H(\varphi), \qquad S = -\CC \sigma H(\varphi),
	\end{equation}
	where $H$ is an interpolation function such that $H(-1)= 0$ and $H(1) = 1$, while the terms $H(\varphi)\PP \sigma$, $H(\varphi) \AA$ and $H(\varphi) \CC \sigma$ model, respectively, the proliferation of tumor cells proportional to nutrient concentration, apoptosis of tumor cells, and consumption of nutrient by the tumor cells.
	
	More sophisticated models include the cell velocities satisfying a generalized Darcy's (or Brinkman's) law where, besides the pressure gradient, it also appears the so-called Korteweg force due to cell concentration or/and multiphase order parameters differentiating, e.g., between necrotic and proliferating tumors (cf., e.g., \cite{FLRS, KS}).\medskip  
	
	While there exist quite a number of numerical simulations of diffuse-interface models of tumor growth (cf., e. g., \cite[Ch.~8]{CL10}, \cite{CLW09, HZO11, WLFC08}), there are less contributions to the mathematical analysis of the models. The first contributions in this direction dealt with the case where the nutrient is neglected, which then leads to the so-called Cahn-Hilliard-Hele-Shaw system (see, e.g., \cite{GRS, JWZ15}). Later on, in  \cite{FriGrasRoc}, the model \eqref{diffuse} including the evolution of the nutrient proportion in case \eqref{lin} (introduced in \cite{HZO11}) is taken into account and rigorously analyzed concerning well-posedness, regularity, and asymptotic behavior. We also refer to the papers \cite{CGRS1, CGRS2}, in which various viscous approximations of the system \eqref{diffuse} have been studied analytically and to \cite{CavRocWu} where optimal control and long-time behaviour of solutions were investigated. Regarding instead the well-posedness of system \eqref{diffuse}-\eqref{phe}, we can quote, e.g.,  the papers \cite{GarLam} as well as \cite{GarLamRoc} and \cite{MRS} where also the optimal control problem and the long-time behavior of solutions in terms of attractors have been tackled. 
	
	Furthermore, although the diffuse-interface approach has proved to be a powerful way to describe tumor growth evolution, a major and challenging issue, still open in the general case, consists in the analytical validation of the transition from diffuse to sharp interfaces (i.e., the limit $\varepsilon\searrow 0$ of \eqref{diffuse}). This will be the subject of the present paper. In the literature some formal results regarding passages to the sharp 
	interface limit are available (cf., e.g., \cite{GarLamSitSty, Hil}), but, up to our knowledge, no rigorous theorems are 
	proved for coupled systems as \eqref{diffuse}. Indeed, in the papers \cite{DFRS15, MelRoc} and \cite{RocSca} the authors investigated the existence of weak solutions and some rigorous sharp interface limit in 
	two simplified cases. In \cite{DFRS15, MelRoc} only the coupling between the Cahn--Hilliard equation 
	and the Darcy law for the velocities is considered and, in particular, in \cite{DFRS15} the physically meaningful case of a double-well potential in the Cahn--Hilliard equation cannot be accounted. In \cite{RocSca} instead, the coupling between the equations for $\varphi$ and $\sigma$ is 
	considered, but in a special and artificial case leading to a gradient flow structure: thanks to this property, the authors of~\cite{RocSca} can prove the convergence result with the tools of Gamma Convergence.
	
	There are, indeed, basically two tested methods in the literature related to the issue of performing a rigorous sharp interface limit, in the spirit of what is already known for the Cahn-Hilliard equation (cf., e.g., \cite{Chen,Le, T05} and references therein). The first approach consists in writing down the system as a Gradient Flow in order to use refined results of Gamma Convergence already exploited in \cite{S11}. This result, however, does not apply to models coupling phase and nutrient dynamics relevant in applications and, as previously mentioned, it has been obtained in \cite{RocSca} only for a toy model suitably adapted in order to possess a gradient flow structure. We also quote \cite{KroemLaux}, \cite{LauxStinson} and \cite{LauxSim} (see also \cite{HensStin}) where the same approach has been used for a Cahn-Hilliard equation with non-constant mobility, a model of lithium-ion batteries, and for the vectorial Allen-Cahn equation, respectively. A second possibility relies in considering a weak notion of solution for the sharp-interface problem, called varifold solution and introduced for the Cahn-Hilliard equation in \cite{Chen}. This second approach has been recently extended in \cite{MelRoc} to the case of a Cahn-Hilliard-Darcy system (first neglecting the nutrient) in the spirit of \cite{AL14}.
	
	We finally mention the very recent technique via relative entropy method, which for the moment is not available for Cahn-Hilliard-type equations but is just limited to the Allen-Cahn equation and its variants \cite{FLS, HensLauxentropy, HensLaux}. The same approach is also used in \cite{FHLS, Laux} for studying stability properties of mean curvature flows.\medskip
	
	In this paper we adopt the second method via varifold solutions and we address the problem of a rigorous sharp interface limit for the following Cahn-Hilliard-Nutrient systems presenting two different source terms.  
	We let $\Omega$ be a (sufficiently) smooth bounded domain in $\R^d$, with $d=2$ or $d=3$, and given a time horizon $T>0$ we set $Q:=(0,T)\times \Omega$. We also denote by $n$ the outward unit normal on the boundary $\partial\Omega$. Given a small parameter $\eps>0$, for $i=1,2$ the two diffuse interface problems we consider are given by
	\begin{equation}\label{eq:diffusemodel1}
		\begin{cases}
			\dot\varphi^\varepsilon-\Delta\mu^\varepsilon=R_i(\varphi^\eps,\sigma^\eps,\mu^\eps), \quad \mu^\varepsilon =-\varepsilon \Delta\varphi^\varepsilon+\frac 1\varepsilon F'(\varphi^\varepsilon),&\text{in }Q,\\
			\dot\sigma^\varepsilon-\Delta\sigma^\varepsilon=S_i(\varphi^\eps,\sigma^\eps,\mu^\eps),&\text{in }Q,\\
		\end{cases}
	\end{equation}
	where the potential $F$ is of double-well type (the prototypical example is given by \eqref{double}), while $R_i,S_i$ are specific coupling terms. Referring to the expressions \eqref{lin} and \eqref{phe}, we focus, for simplicity but without any loss of generality, our attention on the choices 
	\begin{subequations}
		\begin{equation}\label{eq:problemP}
			R_1(\varphi,\sigma,\mu)=-S_1(\varphi,\sigma,\mu)=P(\varphi)(\sigma-\mu),
		\end{equation}
		and
		\begin{equation}\label{eq:problemH}
			R_2(\varphi,\sigma,\mu)=(\sigma-1)H(\varphi),\qquad S_2(\varphi,\sigma,\mu)=-\sigma H(\varphi),
		\end{equation}
	\end{subequations}
	where $P$ is a so-called proliferation function and  $H$ is an interpolation function (cf.~\eqref{lin} and \eqref{phe}). Moreover, system \eqref{eq:diffusemodel1} is endowed with the following Neumann-homogeneous boundary conditions, ensuring no flux flows through the boundary of $\Omega$, and initial conditions:
	\begin{equation}\label{eq:diffusemodel2}
		\begin{cases}		\partial_n\varphi^\varepsilon=\partial_n\sigma^\varepsilon=\partial_n\mu^\varepsilon=0,&\text{ in }(0,T)\times \partial\Omega,\\		\varphi^\varepsilon(0)=\varphi^\varepsilon_0,\quad\sigma^\varepsilon(0)=\sigma^\varepsilon_0,&\text{ in }\Omega.
		\end{cases}
	\end{equation}
	
	Throughout the paper, we will refer to system \eqref{eq:diffusemodel1}--\eqref{eq:diffusemodel2} with choices \eqref{eq:problemP} and \eqref{eq:problemH} as Problem P and Problem H, respectively.\medskip
	
	The sharp interface limit, namely the limit as $\varepsilon\to 0$, of both Problems P and H has been formally obtained in \cite{GarLamSitSty} via matched asymptotic expansion, and it consists in the following Mullins-Sekerka type geometric flow
	
	\begin{equation}\label{eq:sharpinterfacemodel}
		\begin{cases}
			\varphi=1,&\text{in } Q^C,\\
			\varphi=-1,&\text{in } Q^H,\\
			-\Delta\mu=R_i(\varphi,\sigma,\mu), &\text{in } Q^C\cup Q^H,\\
			\dot\sigma-\Delta\sigma=S_i(\varphi,\sigma,\mu), &\text{in } Q^C\cup Q^H,\\
			\mu=\theta \kappa,&\text{in } \Gamma,\\
			0= \llbracket\sigma\rrbracket^C_H,&\text{in } \Gamma,\\ 
			\omega=-\frac 12 \llbracket\partial_n\mu\rrbracket^C_H,&\text{in } \Gamma,\\
			0= \llbracket\partial_n\sigma\rrbracket^C_H,&\text{in } \Gamma,\\        \partial_n\sigma=\partial_n\mu=0,&\text{in }(0,T)\times \partial\Omega,\\
			\varphi(0)=-1+2\chi_{\Omega_0},\quad \sigma(0)=\sigma_0,&\text{in }\Omega.
		\end{cases}
	\end{equation}
	In the above system, $Q^C$ and $Q^H$ are open subsets of $Q$ representing the cancerous and the healthy zone, respectively. The interface between $Q^C$ and $Q^H$ is denoted by $\Gamma$, whose slices $\Gamma_t$ have mean curvature $\kappa(t)$ and scalar normal velocity $\omega(t)$, with the convention that positive normal velocity points towards the healthy region. With a little abuse of notation we still use the symbol $n$ for the normal unit vector of $\Gamma_t$, pointing towards the healthy region. The constant $\theta$ is given by 
	\begin{equation}\label{eq:theta}
		\theta:=\int_{-1}^1\sqrt{\frac{F(u)}{2}}\d u,
	\end{equation}
	and, finally, the symbol $\llbracket\alpha\rrbracket^C_H$ stands for the jump through $\Gamma$ of the quantity $\alpha$ from $\Omega^C$ to $\Omega^H$. The initial data consist in the initial cancerous zone $\Omega_0\subseteq \Omega$ and in the initial nutrient concentration $\sigma_0$.
	
	We mention that more general models than \eqref{eq:sharpinterfacemodel} have been formally obtained in \cite{GarLamSitSty}, even including chemotaxis effects and active transport. However, the approach we follow in this paper strictly relies on the form of the chemical potential $\mu^\eps$ in \eqref{eq:diffusemodel1}. Hence the presence of chemotaxis, which directly affects $\mu^\eps$, can not be directly handled via our method, but probably requires other techniques. We leave this clearly important issue open for future investigations.

	In the present paper, we rigorously prove that solutions to Problems P and H converge, as $\eps$ goes to $0$, to varifold solutions \cite{Chen} of the sharp interface models \eqref{eq:sharpinterfacemodel}. However, we stress that in order to complete the argument for Problem H we need to add a technical assumption on the function $H$ (see \eqref{hyp:Htechnical}). Although this assumption still allows to include many nonlinearities, in particular it yields $H(1)=0$ and so it excludes many prototypical examples for which $H(1)=1$. In comparison with existing contributions, the main difficulty we encounter here relies on the fact that the uniform bounds on the chemical potential $\mu^\varepsilon$, that were derived in \cite{Chen} for the classical Cahn-Hilliard equation, cannot be directly applied due to the presence of mass sources in \eqref{eq:diffusemodel1}. This is the main reason why our result is not global in time anymore (cfr.~Theorems~\ref{thm:main}, \ref{thm:mainH}) like it was in \cite{Chen, MelRoc}, but it holds true only till when the interface is effectively present and we do not reach the pure phases. However, we think that this is the most interesting case to be studied and in this sense our result should not be regarded as a partial one.
	
	The two cases described by the different source terms $R_i$ and $S_i$ must be treated differently, basically because in case $i=1$ we can rely on an energy balance featuring a Lyapunov functional which is not available in case $i=2$. To overcome the issue, in this second case we carefully apply a Gr\"onwall type argument (this is where we need the technical assumption \eqref{hyp:Htechnical} on $H$) in order to obtain proper uniform bounds and conclude the argument.  
	\bigskip
	
	\noindent\textbf{Plan of the paper.} In Section~\ref{sec:setting} we first list all the assumptions we require for the sharp interface limit analysis. We then present the known well-posedness results for the diffuse interface problem \eqref{eq:diffusemodel1}--\eqref{eq:diffusemodel2} with choices \eqref{eq:problemP} and \eqref{eq:problemH} (Problems P and H), see Theorems~\ref{thm:existenceP} and \ref{thm:existenceH} respectively, and we introduce in Definition~\ref{def:varifoldsol} the weak notion of varifold solution for the sharp interface problem \eqref{eq:sharpinterfacemodel}. We finally state the main results of the paper regarding the convergence of solutions of Problems P and H to varifold solutions of \eqref{eq:sharpinterfacemodel} in Theorems~\ref{thm:main} and \ref{thm:mainH}. The subsequent Sections~\ref{sec:sharpinterfaceP} and \ref{sec:sharpinterfaceH} are devoted to the proof of such results. The strategy exploits energetic arguments to deduce uniform bounds on both the phase variable, the nutrient and the chemical potential, leading to suitable compactness properties. The validity of the sharp interface limit is finally obtained by a careful construction of a proper varifold.

	\section*{Notation and preliminaries}
	
	The maximum and the minimum of two real numbers $\alpha,\beta$ are denoted by $\alpha\vee\beta$ and $\alpha\wedge\beta$, respectively. The positive part of a function $f$ is $f^+:=f\vee0$.
	
	We adopt standard notations for Lebesgue, Sobolev and Bochner spaces. The norm of a Banach space $X$ will be denoted by $\|\cdot\|_X$; if $X$ is an Hilbert space, its scalar product will be written $(\cdot,\cdot)_X$. In the case $X=L^p(\Omega)$, we use the shortcuts $\|\cdot\|_p$ and $(\cdot,\cdot)_2$, if $p=2$. The duality product between $w\in X^*$ and $v\in X$ is represented by the symbol $\langle w,v\rangle_X$. Strong, weak and weak$^*$ convergence are denoted by $\xrightarrow{}$, $\xrightharpoonup[]{}$ and $\xrightharpoonup[]{\ast}$, respectively.
	
	With the symbol $B([a,b];X)$ we mean the space of everywhere defined functions $f\colon [a,b]\to X$ which are bounded in $X$, namely $\sup\limits_{t\in [a,b]}\|f(t)\|_X<+\infty$. The spaces of continuous functions, $\alpha$-H\"older continuous functions and functions of bounded variation from $[a,b]$ to $X$ are instead denoted by $C([a,b];X)$, $C^\alpha([a,b];X)$ and $BV([a,b];X)$, respectively. By $L^p_{\rm loc}([0,T);X)$ we denote the space of functions which belong to $L^p(0,t;X)$ for all $t\in (0,T)$. Finally, for an arbitrary subset $E$ of $\R^d$ (not necessarily open), the symbols $C_0(E)$ and $C_0(E; \R^d)$ stand for the spaces of scalar and vector-valued continuous functions defined in $E$ whose support is compact.
	
	The average of a function $f\colon \Omega\to \R$, where $\Omega$ is a measurable subset of $\R^d$, is denoted by $\displaystyle[f]:=\frac{1}{|\Omega|}\int_\Omega f\d x$.
	
	Given a set $A\subseteq (0,T)\times \Omega$, its slice with respect to time $t\in(0,T)$ will be indicated by $A_t:=\{x\in\Omega:\, (t,x)\in A\}$. Moreover, the Lebesgue measure in $[0,T]\times \R^d$ is denoted by $\d t\d x$.
	
	\bigskip
	
	\noindent\textbf{Varifolds.}
	A varifold $\mc V$ on an open set $\Omega\subseteq \R^d$ is a Radon measure on $\Omega\times \mc P$, where $\mc P$ denotes the set of unit normals of unoriented hyperplanes in $\R^d$, namely
	\begin{equation*}
		\mc P:=\mathbb{S}^{d-1}/\sim_{\mathbb S^{d-1}},
	\end{equation*}
	where $\sim_{\mathbb S^{d-1}}$ denotes the antipodal equivalence relation in $\mathbb{S}^{d-1}$.
	
	The first variation $\delta\mc V$ of a varifold $\mc V$ is the linear functional on $C^1_0(\Omega;\R^d)$ defined by
	\begin{equation}\label{eq:firstvar}
		\langle\delta\mc V, Y\rangle:=\int_{\Omega\times\mc P}\nabla Y(x):(I-p\otimes p)\d \mc V(x,p),\quad\text{for any }Y\in C^1_0(\Omega;\R^d).
	\end{equation}
	
	\section{Setting and main results}\label{sec:setting}
	The ambient space $\Omega$ is a bounded domain in $\R^d$, with $d=2$ or $d=3$, whose boundary $\partial \Omega$ is of class $C^{2,1}$. We recall that, given $T>0$, we set $Q:=(0,T)\times \Omega$. We require that the potential $F\in C^3(\R)$ is a nonnegative function vanishing only at the points $1$ and $-1$ and satisfying
	\begin{equation}\label{hyp:F}
		F''(u)\ge c|u|^{p-2}\,\,\text{ if }|u|\ge 1-\delta_0,\quad \text{for some }\delta_0\in (0,1)\text{ and } p\in [3,6).
	\end{equation}
	
	We assume that the proliferation function $P\in C^{0,1}_{\rm loc}(\R)$ appearing in \eqref{eq:problemP} is nonnegative and fulfils
	\begin{equation}\label{hyp:P}
		|P'(u)|\le C(1+|u|^{r-1}) \text{ for a.e. }u\in\R, \quad\text{for some }r\in [1,p-2].
	\end{equation}
	Examples of proliferation functions are $P(u)=\lambda_0 (1+u)^+$, see \cite{GarLamSitSty}, or $P(u)=(1-u^2)\vee \lambda_0(|u|-1)$, see \cite{FriGrasRoc}, where $\lambda_0>0$ is a positive parameter. 
	
	Notice that assumptions \eqref{hyp:F} and \eqref{hyp:P} imply the existence of constants $c_F, C_F, \overline{c}_F, C_P, \overline{C}_P>0$ such that
	\begin{subequations}
		\begin{align}
			& F(u)\ge c_F|u|^p-C_F,&&\text{for every }u\in\R;\label{eq:propF1}\\
			& F(u)\ge \overline{c}_F(|u|-1)^2,&&\text{for every }u\in\R;\label{eq:propF2}\\
			& P(u)\le C_P(1+|u|^r),&&\text{for every }u\in\R;\label{eq:propP1}\\
			& |P(u)-P(v)|\le \overline{C}_P|u-v|(1+|u|^{r-1}+|v|^{r-1}),&&\text{for every }u,v\in\R.\label{eq:propP2}
		\end{align}
	\end{subequations}
	
	As regards \eqref{eq:problemH} instead, we assume that the interpolation function $H$ possesses the following properties:
	\begin{subequations}
		\begin{gather}
			\label{hyp:H}
			H\colon \R \to [0,1]\text{ is Lipschitz continuous and of class $C^2(\R)$;}\\
			 H(u)\le C\frac{F(u)}{|F'(u)|},\quad\text{for all $u$ satisfying }F'(u)<0.\label{hyp:Htechnical}
		\end{gather}   
	\end{subequations}
	A standard example of interpolation function is (a regularization of) ${H}(u)=1\wedge \left(\frac{1+u}{2}\right)^+$, see for instance \cite{CWSL14, GarLamSitSty}. However, note that this example does not fulfil the technical assumption \eqref{hyp:Htechnical} (indeed, inequality \eqref{hyp:Htechnical} yields $H(\pm 1)=0$), which anyway is crucial to complete the sharp interface limit argument (see Lemma \ref{lemma:technical}). Nevertheless, a large class of nonlinearities satisfies both \eqref{hyp:H} and \eqref{hyp:Htechnical}.
	
	As customary in the analysis of the Cahn-Hilliard equation, we finally introduce the natural Ginzburg-Landau energy
	\begin{equation*}
		E^\eps(t):=\int_\Omega e^\eps(\varphi^\eps(t))\d x,\qquad \text{for }t\in [0,T],
	\end{equation*}
	where we define the associated energy density
	\begin{equation}\label{eq:endens}
		e^\eps(\varphi):=\frac\eps 2|\nabla\varphi|^2+\frac 1\eps F(\varphi).
	\end{equation}
	
	\subsection{Diffuse interface models}
	Existence of \emph{strong} solutions to problem \eqref{eq:diffusemodel1}--\eqref{eq:diffusemodel2} with the choice \eqref{eq:problemP}, here and henceforth called Problem P, has been proved in \cite{FriGrasRoc} under slightly stronger assumptions on the function $F$, which however still include the classical double-well potential \eqref{double}.
	\begin{thm}[\textbf{Well-posedness for Problem P}]\label{thm:existenceP}
		In addition to \eqref{hyp:F}, assume that the potential $F$ can be written as $F=F_{\rm c}+F_{\rm nc}$, and $F_{\rm c},F_{\rm nc}\in C^2(\R)$ satisfy
		\begin{equation}\label{hyp:F2}
			\begin{aligned}
				&c(1+|u|^{p-2})\le F_{\rm c}''(u)\le C(1+|u|^{p-2}),&&\text{for all }u\in \R;\\
				&|F_{\rm nc}''(u)|\le C,&&\text{for all }u\in \R.
			\end{aligned}
		\end{equation}
		Let the proliferation function $P$ satisfy \eqref{hyp:P}. Then for any pair of initial data $(\varphi^\eps_0,\sigma^\eps_0)\in H^3(\Omega)\times H^1(\Omega)$ with $\partial_n\varphi^\eps_0=0$ in $\partial \Omega$ there exists a unique solution $(\varphi^\eps,\sigma^\eps)$ to Problem P in the sense that:
		\begin{itemize}
			\item[(1)] $
			\varphi^\eps\in L^\infty(0,T;H^3(\Omega))\cap H^1(0,T;H^1(\Omega)), \quad\text{with}\quad \partial_n\varphi^\eps=0\text{ in }L^\infty(0,T;L^2(\partial\Omega));\smallskip\\
			\sigma^\eps\in L^\infty(0,T;H^1(\Omega))\cap H^1(0,T;L^2(\Omega));\smallskip\\
			\mu^\eps:= -\eps\Delta\varphi^\eps+\frac 1\eps F'(\varphi^\eps)\in L^\infty(0,T;H^1(\Omega));
			$\smallskip
			\item[(2)] $\varphi^\eps(0)=\varphi^\eps_0$ and  $\sigma^\eps(0)=\sigma^\eps_0$;\smallskip
			\item[(3)] for almost every time $t\in (0,T)$ and for all $\Phi,\Psi\in H^1(\Omega)$ there hold
			\begin{equation}\label{eq:weakeq}
				\begin{cases}
					(\dot\varphi^\eps(t),\Phi)_{2}+(\nabla\mu^\eps(t),\nabla \Phi)_{2}=(P(\varphi^\eps(t))(\sigma^\eps(t)-\mu^\eps(t)), \Phi)_{2},\\
					(\dot\sigma^\eps(t),\Psi)_{2}+(\nabla\sigma^\eps(t),\nabla \Psi)_{2}=-(P(\varphi^\eps(t))(\sigma^\eps(t)-\mu^\eps(t)), \Psi)_{2};
				\end{cases}
			\end{equation}
			\item[(4)] For all $t\in [0,T]$ the following energy balance holds true:
			\begin{equation}\label{eq:EB}
				\begin{aligned}
					&\,E^\eps(t)+\frac 12 \|\sigma^\eps(t)\|^2_{2}+\int_0^t\left(\|\nabla\mu^\eps(\tau)\|^2_{2}+\|\nabla\sigma^\eps(\tau)\|^2_{2}+\|\sqrt{P(\varphi^\eps(\tau))}(\sigma^\eps(\tau){-}\mu^\eps(\tau))\|^2_{2}\right)\d\tau\\
					&\,= E^\eps(0)+\frac 12 \|\sigma^\eps_0\|^2_{2}.
				\end{aligned}
			\end{equation}
		\end{itemize}
	\end{thm}
	
	\begin{rmk}
		For almost every time $t\in [0,T]$, since $\varphi^\eps(t)\in H^3(\Omega)$ and $d=2,3$, Sobolev embedding yields $\varphi^\eps(t)\in L^\infty(\Omega)$, whence $P(\varphi^\eps(t))(\sigma^\eps(t)-\mu^\eps(t))\in L^2(\Omega)$ and so the right-hand side of \eqref{eq:weakeq} is meaningful.
		
		Since $\varphi^\eps\in C([0,T];H^1(\Omega))$, again by Sobolev embedding one deduces $\varphi^\eps\in C([0,T];L^6(\Omega))$. Recalling that $F$ has at most $6$-growth by \eqref{hyp:F2}, we thus obtain $F(\varphi^\eps)\in C([0,T];L^1(\Omega))$, and so the energy $E^\eps$ is continuous.
		
		Since $\varphi^\eps,\sigma^\eps,\mu^\eps\in L^\infty(0,T;H^1(\Omega))$, as before we infer $\varphi^\eps,\sigma^\eps,\mu^\eps\in L^\infty(0,T;L^6(\Omega))$. By using the 4-growth of $P$ (see \eqref{hyp:P}), this implies $\sqrt{P(\varphi^\eps(\tau))}(\sigma^\eps(\tau){-}\mu^\eps(\tau))\in L^\infty(0,T;L^2(\Omega))$ and so the energy balance \eqref{eq:EB} makes sense.
		
		Since $\dot\varphi^\eps,\dot\sigma^\eps\in L^2(0,T;L^2(\Omega))$, standard elliptic regularity yields $\sigma^\eps,\mu^\eps\in L^2(0,T;H^2(\Omega))$, and so the solution provided by Theorem~\ref{thm:existenceP} is actually a strong solution to \eqref{eq:diffusemodel1}.
	\end{rmk}
	
	\begin{rmk}
		Observe that solutions to Problem P preserve the mass of the sum of phase variable and nutrient, namely
		\begin{equation}\label{eq:massconservation}
			[\varphi^\eps(t)+\sigma^\eps(t)]=[\varphi^\eps_0+\sigma^\eps_0],\qquad\text{for all }t\in [0,T].
		\end{equation}
		Indeed, by summing the two lines in \eqref{eq:weakeq}, and choosing $\Phi=\Psi\equiv1$, for almost every time we obtain
		\begin{equation*}
			\frac{\d}{\d t}[\varphi^\eps(t)+\sigma^\eps(t)]=\frac{(\dot\varphi^\eps(t),1)_{2}+(\dot\sigma^\eps(t),1)_{2}}{|\Omega|}=0.
		\end{equation*}
	\end{rmk}
	
	Existence of \emph{strong} solutions to problem \eqref{eq:diffusemodel1}--\eqref{eq:diffusemodel2} for the choice \eqref{eq:problemH}, here and henceforth called Problem H, has instead been proved in \cite{GarLamRoc} (see also \cite{GarLam}), again with stricter hypotheses on $F$. In this case, exploiting the maximum principle, the authors showed that the nutrient remains confined between $0$ and $1$ along the entire evolution, starting of course from an initial state fulfilling the same condition.
	\begin{thm}[\textbf{Well-posedness for Problem H}]\label{thm:existenceH}
		In addition to \eqref{hyp:F}, assume that the potential $F$ fulfils
		\begin{equation}\label{hyp:F3}
			\begin{aligned}
				&|F'(u)|\le C(1+F(u)),&&\text{for all }u\in \R;\\
				&|F''(u)|\le C(1+|u|^{p-2}),&&\text{for all }u\in \R.
			\end{aligned}
		\end{equation}
		Let the interpolation function $H$ satisfy \eqref{hyp:H}. Then for any pair of initial data $(\varphi^\eps_0,\sigma^\eps_0)\in H^3(\Omega)\times H^1(\Omega)$ with $\partial_n\varphi^\eps_0=0$ in $\partial \Omega$ and $0\le \sigma^\eps_0\le 1$ almost everywhere in $\Omega$ there exists a unique solution $(\varphi^\eps,\sigma^\eps)$ to Problem H in the sense that:
		\begin{itemize}
			\item[(1)] $
			\varphi^\eps\in L^\infty(0,T;H^2(\Omega))\cap L^2(0,T;H^3(\Omega))\cap H^1(0,T;L^2(\Omega))\cap C(\overline Q),\\\text{with } \partial_n\varphi^\eps=0\text{ in }L^\infty(0,T;L^2(\partial\Omega));\smallskip\\
			\sigma^\eps\in L^\infty(0,T;H^1(\Omega))\cap L^2(0,T;H^2(\Omega))\cap H^1(0,T;L^2(\Omega)),\\ \text{with } 0\le \sigma^\eps\le 1 \text{ almost everywhere in $Q$};\smallskip\\
			\mu^\eps:= -\eps\Delta\varphi^\eps+\frac 1\eps F'(\varphi^\eps)\in L^\infty(0,T;L^2(\Omega))\cap L^2(0,T;H^2(\Omega));
			$\smallskip
			\item[(2)] $\varphi^\eps(0)=\varphi^\eps_0$ and  $\sigma^\eps(0)=\sigma^\eps_0$;\smallskip
			\item[(3)] for almost every time $t\in (0,T)$ and for all $\Phi,\Psi\in H^1(\Omega)$ there hold
			\begin{equation}\label{eq:weakeqH}
				\begin{cases}
					(\dot\varphi^\eps(t),\Phi)_{2}+(\nabla\mu^\eps(t),\nabla \Phi)_{2}=((\sigma^\eps(t)-1)H(\varphi^\eps(t)), \Phi)_{2},\\
					(\dot\sigma^\eps(t),\Psi)_{2}+(\nabla\sigma^\eps(t),\nabla \Psi)_{2}=-(\sigma^\eps(t)H(\varphi^\eps(t)), \Psi)_{2};
				\end{cases}
			\end{equation}
			\item[(4)] For all $t\in [0,T]$ the following energy balance holds true:
			\begin{equation}\label{eq:EBH}
				\begin{aligned}
					&\,E^\eps(t)+\frac 12 \|\sigma^\eps(t)\|^2_{2}+\int_0^t\left(\|\nabla\mu^\eps(\tau)\|^2_{2}+\|\nabla\sigma^\eps(\tau)\|^2_{2}+\|\sqrt{H(\varphi^\eps(\tau))}\sigma^\eps(\tau)\|^2_{2}\right)\d\tau\\
					&\,=\, E^\eps(0)+\frac 12 \|\sigma^\eps_0\|^2_{2}+\int_0^t(\mu^\eps(\tau),(\sigma^\eps(\tau)-1)H(\varphi^\eps(\tau))_2\d\tau.
				\end{aligned}
			\end{equation}
		\end{itemize}
	\end{thm}
	
	\begin{rmk}
		The validity of the energy balance $(4)$ above is not proved in \cite{GarLamRoc}, but can be deduced from the regularity of $\varphi^\eps$ given by $(1)$. We refer for instance to \cite[Lemma~3.9]{RocSchim}.
	\end{rmk}
	
	\begin{rmk}\label{rmk:Lip}
		Notice that for solutions to Problem H the map $t\mapsto [\varphi^\eps(t)]$ is nonincreasing and 1-Lipschitz continuous in $[0,T]$. Indeed from \eqref{eq:weakeqH}$_1$, exploiting that both $H$ and $\sigma^\eps$ are nonnegative and bounded by $1$ we deduce 
		\begin{equation*}
			\frac{\d}{\d t}[\varphi^\eps(t)]=\frac{1}{|\Omega|}((\sigma^\eps(t)-1)H(\varphi^\eps(t)), 1)_{2}\in [-1,0].
		\end{equation*}
	\end{rmk}
	
	\subsection{Sharp interface models}	
	The notion of solution to the sharp interface model \eqref{eq:sharpinterfacemodel} we consider in this paper is the weak notion of \emph{varifold solution} introduced for the Cahn-Hilliard equation in \cite{Chen}, to which we refer for a complete explanation (see also \cite{MelRoc}). We just stress that the use of varifolds as in $(c)$ and $(e)$ below is a way to write in a weak sense the equation 
	\begin{equation*}
		\mu=\theta\kappa,\qquad \text{in }\Gamma,
	\end{equation*}
	the constant $\theta$ being defined in \eqref{eq:theta}. The energy inequality $(f)$ in the following Def.~\ref{def:varifoldsol} is instead interpreted as an entropy condition, which selects feasible solutions.
	
	\begin{defi}\label{def:varifoldsol}
		Given an initial measurable set $\Omega_0\subseteq \Omega$ and an initial datum $\sigma_0\in L^2(\Omega)$, we say that a quadruple $(Q^C, \sigma, \mu,V)$ is a \emph{varifold solution} to the sharp interface model \eqref{eq:sharpinterfacemodel}, with either \eqref{eq:problemP} or \eqref{eq:problemH}, in $[0,T]$ if the following conditions are satisfied:
		\begin{itemize}
			\item[(a)] $\sigma,\mu\in L^2(0,T;H^1(\Omega))$;
			\item[(b)] the set $Q^C\subseteq Q$ fulfils $\chi_{Q^C}\in C([0,T];L^1(\Omega))\cap B([0,T];BV(\Omega))$;
			\item[(c)] $V$ is a Radon measure on $Q\times\mc P$ with the following property: for almost every $t\in (0,T)$ there exist a Radon measure $\lambda^t$ on $\overline\Omega$ and $\lambda^t$-measurable functions $c_i^t\colon \overline{\Omega}\to \R$, $p_i^t\colon \overline{\Omega}\to \mc P$, for $i=1,\dots, d$, fulfilling
			\begin{align}
				&0\le c_i^t\le 1,\quad \sum_{i=1}^{d}c_i^t\ge 1,\quad \sum_{i=1}^{d}p_i^t\otimes p_i^t=I,\quad \lambda^t\text{-}a.e.,\label{eq:coefficients}\\
				& \frac{\d|D\chi_{Q^C_t}|}{\d\lambda^t}\le \frac{1}{2\theta},\quad \lambda^t\text{-}a.e.,\label{eq:density}
			\end{align}
			where $\theta$ is given by \eqref{eq:theta}, such that, defining the varifold $V^t$ on $\Omega$ as
			\begin{equation}\label{eq:defVt}
				\int_{\Omega\times \mc P}\psi(x,p)\d V^t(x,p):=\sum_{i=1}^d\int_\Omega c_i^t(x)\psi(x,p_i^t(x))\d \lambda^t(x),\quad\text{for every }\psi\in C_0(\Omega\times\mc P),
			\end{equation}
			there holds $\d V(t,x,p)=\d V^t(x,p)\d t$;
			\item[(d)] For any $\Phi,\Psi\in L^2(0,T;H^1(\Omega))\cap H^1(0,T;L^2(\Omega))$ with $\Phi(T)=\Psi(T)=0$, setting $\varphi:=-1+2\chi_{Q^C}$ there holds
			\begin{equation*}
				\begin{cases}\displaystyle
					\int_0^T\!\!\left({-}2\int_{Q^C_\tau}\dot\Phi(\tau)\d x{+}(\nabla\mu(\tau),\nabla\Phi(\tau))_{2}{-}(P(\varphi(\tau))(\sigma(\tau){-}\mu(\tau)),\Phi(\tau))_{2}\right)\d\tau=2\!\int_{\Omega_0}\!\!\!\Phi(0)\d x,\\\displaystyle
					\int_0^T\left({-}(\sigma(\tau),\dot\Psi(\tau))_{2}{+}(\nabla\sigma(\tau),\nabla\Psi(\tau))_{2}{+}(P(\varphi(\tau))(\sigma(\tau){-}\mu(\tau)),\Psi(\tau))_{2}\right)\d\tau=(\sigma_0,\Psi(0))_{2},
				\end{cases}
			\end{equation*}
			as regards problem \eqref{eq:problemP}, while there holds
			\begin{equation*}
				\begin{cases}\displaystyle
					\int_0^T\left({-}2\int_{Q^C_\tau}\dot\Phi(\tau)\d x{+}(\nabla\mu(\tau),\nabla\Phi(\tau))_{2}{-}((\sigma(\tau){-}1)H(\varphi(\tau)),\Phi(\tau))_{2}\right)\d\tau=2\int_{\Omega_0}\Phi(0)\d x,\\\displaystyle
					\int_0^T\left({-}(\sigma(\tau),\dot\Psi(\tau))_{2}{+}(\nabla\sigma(\tau),\nabla\Psi(\tau))_{2}{+}(\sigma(\tau)H(\varphi(\tau)),\Psi(\tau))_{2}\right)\d\tau=(\sigma_0,\Psi(0))_{2},
				\end{cases}
			\end{equation*}
			in case \eqref{eq:problemH};
			\item[(e)] for almost every $t\in(0,T)$ one has
			\begin{equation*}
				\langle\delta V^t, Y\rangle=2\int_{Q^C_t}\div(\mu(t)Y)\d x,\quad\text{for all }Y\in C^1_0(\Omega;\R^d),
			\end{equation*}
			where the first variation of a varifold has been introduced in \eqref{eq:firstvar};
			\item[(f)] for almost every $s,t\in (0,T)$ with $s\le t$ the following two energy inequalities hold true for \eqref{eq:problemP} and \eqref{eq:problemH}, respectively:
			\begin{equation*}
				\begin{aligned}
					1)\quad&\lambda^t(\overline{\Omega})+\frac 12 \|\sigma(t)\|^2_{2}+\int_s^t\left(\|\nabla\mu(\tau)\|^2_{2}+\|\nabla\sigma(\tau)\|^2_{2}+\|\sqrt{P(\varphi(\tau))}(\sigma(\tau){-}\mu(\tau))\|^2_{2}\right)\d\tau\\
					\le& \lambda^s(\overline{\Omega})+\frac 12 \|\sigma(s)\|^2_{2}.\bigskip\\	
					2)\quad&\lambda^t(\overline{\Omega})+\frac 12 \|\sigma(t)\|^2_{2}+\int_s^t\left(\|\nabla\mu(\tau)\|^2_{2}+\|\nabla\sigma(\tau)\|^2_{2}+\|\sqrt{H(\varphi(\tau))}\sigma(\tau))\|^2_{2}\right)\d\tau\\
					\le& \lambda^s(\overline{\Omega})+\frac 12 \|\sigma(s)\|^2_{2}+\int_s^t(\mu(\tau),(\sigma(\tau)-1)H(\varphi(\tau))_2\d\tau.
				\end{aligned}
			\end{equation*}
		\end{itemize}
	\end{defi}
	
	\subsection{Main results}
	We are now in the  position to state the main theorems of the paper, regarding the rigorous limit of solutions to the diffuse interface Problems P and H to varifold solutions of the sharp interface problem \eqref{eq:sharpinterfacemodel}. We stress that the validity of such limit holds as long as the resulting evolution has a proper interface, namely as long as the limit phase variable $\varphi$ does not degenerate in a pure phase $\varphi\equiv \pm 1$ (see \eqref{eq:Tmax} and \eqref{eq:TmaxH}). We recall that pure phases represent a completely cancerous or a completely healthy tissue. The proofs of such results will be given in Sections~\ref{sec:sharpinterfaceP} and \ref{sec:sharpinterfaceH}, respectively.
	\begin{thm}[\textbf{Sharp interface limit for Problem P}]\label{thm:main}
		Assume that the potential $F$ satisfies \eqref{hyp:F}, \eqref{hyp:F2} and that the proliferation function $P$ satisfies \eqref{hyp:P}. Let $(\varphi^\eps,\sigma^\eps)$ be the unique solution to Problem P given by Theorem~\ref{thm:existenceP} with initial data $(\varphi^\eps_0,\sigma^\eps_0)\in H^3(\Omega)\times H^1(\Omega)$ with $\partial_n\varphi^\eps_0=0$ in $\partial \Omega$ satisfying
		\begin{subequations}\label{eq:hypinitial}
			\begin{equation}\label{eq:inenergybound}
				E^\eps(0)+\frac 12\|\sigma^\eps_0\|^2_{2}\le \mc E_0,
			\end{equation}
			and 
			\begin{equation}\label{eq:inmassbound}
				|[\varphi^\eps_0]|\le c_0\in [0,1).
			\end{equation}
		\end{subequations}
		Then, up to passing to a non relabelled subsequence, the following facts hold:
		\begin{itemize}
			
			\item[(I)] there exists a measurable set $Q^C\subseteq Q$ for which
			\begin{align*}
				&\bullet\, \varphi^\eps\xrightarrow[\eps\to 0]{}-1+2\chi_{Q^C}=:\varphi,\qquad\text{in $C([0,T];L^2(\Omega))$};\\
				&\bullet\, \varphi^\eps(t)\xrightarrow[\eps\to 0]{}\varphi(t),\qquad\text{in $L^q(\Omega)$ for all $q\in [1,p)$ and $t\in [0,T]$;}\\
				&\bullet\, \chi_{Q^C}\in C^\frac 18([0,T];L^1(\Omega))\cap B([0,T];BV(\Omega));
			\end{align*}
			\item[(II)] there exists a function $\sigma\in L^2(0,T;H^1(\Omega))$ for which
			\begin{align*}
				&\bullet\, \sigma^\eps \xrightharpoonup[\eps\to 0]{}\sigma,\qquad\text{in }L^2(0,T;H^1(\Omega));\\
				&\bullet\, \sigma^\eps \xrightarrow[\eps\to 0]{}\sigma,\qquad \text{in $L^2(0,T;L^q(\Omega))$ for all $q\in [1,6)$;}
			\end{align*}
			\item[(III)] there exist Radon measures $\lambda, \lambda_{i,j}$ on $\overline Q$, for $i,j=1,\dots,d$, for which
			\begin{align*}
				&\bullet\,e^\eps(\varphi^\eps)\d t\d x\xrightharpoonup[\eps\to 0]{\ast}\lambda,\qquad\qquad\text{in the sense of measures in $\overline Q$};\\
				&\bullet\, \eps\partial_i\varphi^\eps\partial_j\varphi^\eps\d t\d x\xrightharpoonup[\eps\to 0]{\ast}\lambda_{i,j},\quad\,\,\,\text{in the sense of measures in $\overline Q$};
			\end{align*}
		\end{itemize}
		Moreover, defining the critical time as
		\begin{equation}\label{eq:Tmax}
			T_{\rm cr}:=\sup\{t\in (0,T]:\, 0<|Q^C_t|<|\Omega|\},
		\end{equation}
		there hold:
		\begin{itemize}
			\item[(IV)] there exists a function $\mu\in L^2_{\rm loc}([0,T_{\rm cr});H^1(\Omega))$ for which
			\begin{align*}
				&\bullet\, \mu^\eps \xrightharpoonup[\eps\to 0]{}\mu,\quad\text{in }L^2(0,\widetilde{T};H^1(\Omega))\text{ for all $\widetilde{T}\in (0,T_{\rm cr})$};
			\end{align*}
			\item[(V)] there exists a Radon measure $V$ on $(0,T_{\rm cr})\times \Omega\times \mc P$ such that for any $\widetilde{T}\in (0,T_{\rm cr})$ the quadruple $(Q^C, \sigma, \mu,V)$ is a varifold solution to \eqref{eq:sharpinterfacemodel}--\eqref{eq:problemP} in $[0,\widetilde{T}]$ with initial data $(Q^C_0,\sigma_0)$, where $\sigma_0$ is a weak limit of $\sigma^\eps_0$, fulfilling in addition:
			\begin{align*}
				\bullet\,& \d\lambda(t,x)=\d\lambda^t(x)\d t;\\
				\bullet\,& \int_0^{\widetilde{T}}\langle\delta V^\tau,Y(\tau)\rangle\d\tau= \int_0^{\widetilde{T}}\int_\Omega\nabla Y(\tau,x):\left[\d\lambda(\tau,x)I-(\d\lambda_{i,j}(\tau,x))_{d\times d}\right],\\
				&\text{for all }Y\in C^1_0([0,\widetilde T]\times \Omega;\R^d).
			\end{align*}
		\end{itemize}
		Finally, if either
		\begin{subequations}
			\begin{equation}\label{eq:ass1}
				\bullet\, T\mc E_0<\frac{|\Omega|c_F}{2C_P(c_F+C_F)}(1-c_0)^2,
			\end{equation}
			or
			\begin{equation}\label{eq:ass2}
				\bullet\, |[\varphi^\eps_0+\sigma^\eps_0]|\le d_0<1,\quad\text{ and }\quad \mc E_0<\frac{|\Omega|}{2}(1-d_0)^2,
			\end{equation}
		\end{subequations}
		then actually $T_{\rm cr}=T$ and $(Q^C, \sigma, \mu,V)$ is a varifold solution in the whole $[0,{T}]$.
		
	\end{thm}
	
	\begin{rmk}
		We observe that hypotheses \eqref{eq:ass1} and \eqref{eq:ass2} consist in smallness conditions for the initial energy $\mc E_0$ or for the time horizon $T$. As it will be clear in Lemma~\ref{lemma:boundphi}, they are sufficient assumptions in order to prevent the phase variable $\varphi^\eps$ to converge to a pure region in the whole time interval $[0,T]$ (see also \eqref{eq:boundmass}). This explains why in these particular cases it is reasonable to expect $T_{\rm cr}=T$.
	\end{rmk}

	\begin{thm}[\textbf{Sharp interface limit for Problem H}]\label{thm:mainH}
			Assume that the potential $F$ satisfies \eqref{hyp:F}, \eqref{hyp:F3} and that the interpolation function $H$ satisfies \eqref{hyp:H} and \eqref{hyp:Htechnical}. Let $(\varphi^\eps,\sigma^\eps)$ be the unique solution to Problem H given by Theorem~\ref{thm:existenceH} with initial data $(\varphi^\eps_0,\sigma^\eps_0)\in H^3(\Omega)\times H^1(\Omega)$ with $\partial_n\varphi^\eps_0=0$ in $\partial \Omega$ and $0\le \sigma^\eps_0\le 1$ almost everywhere in $\Omega$ satisfying \eqref{eq:hypinitial}. Then, up to passing to a non relabelled subsequence, the following facts hold:
			\begin{itemize}
				
				\item[(I)] there exists a measurable set $Q^C\subseteq Q$ for which
				\begin{align*}
					&\bullet\, \varphi^\eps\xrightarrow[\eps\to 0]{}-1+2\chi_{Q^C}=:\varphi,\qquad\text{in $C([0,T];L^2(\Omega))$};\\
					&\bullet\, \varphi^\eps(t)\xrightarrow[\eps\to 0]{}\varphi(t),\qquad\text{in $L^q(\Omega)$ for all $q\in [1,p)$ and $t\in [0,T]$;}\\
					&\bullet\, \chi_{Q^C}\in C^\frac 18([0,T];L^1(\Omega))\cap B([0,T];BV(\Omega));
				\end{align*}
				\item[(II)] there exists a function $\sigma\in L^2(0,T;H^1(\Omega))$ with $0\le \sigma\le 1$ a.e. in $Q$ for which
				\begin{align*}
					&\bullet\, \sigma^\eps \xrightharpoonup[\eps\to 0]{}\sigma,\quad\text{in }L^2(0,T;H^1(\Omega));\\
					&\bullet\, \sigma^\eps \xrightarrow[\eps\to 0]{}\sigma,\quad \text{in $L^q(Q)$ for all $q\ge 1$;}
				\end{align*}
				\item[(III)] there exist Radon measures $\lambda, \lambda_{i,j}$ on $\overline{Q}$, for $i,j=1,\dots,d$, for which
				\begin{align*}
					&\bullet\,e^\eps(\varphi^\eps)\d t\d x\xrightharpoonup[\eps\to 0]{\ast}\lambda,\qquad\qquad\text{in the sense of measures in $\overline{Q}$};\\
					&\bullet\, \eps\partial_i\varphi^\eps\partial_j\varphi^\eps\d t\d x\xrightharpoonup[\eps\to 0]{\ast}\lambda_{i,j},\quad\,\,\,\text{in the sense of measures in $\overline{Q}$}.
				\end{align*}
			\end{itemize}
			Moreover, defining the critical time as
			\begin{equation}\label{eq:TmaxH}
				T_{\rm cr}:=\sup\{t\in (0,T]:\, 0<|Q^C_t|\},
			\end{equation}
			there hold:
			\begin{itemize}
				\item[(IV)] there exists a function $\mu\in L^2_{\rm loc}([0,T_{\rm cr});H^1(\Omega))$ for which
				\begin{align*}
					&\bullet\, \mu^\eps \xrightharpoonup[\eps\to 0]{}\mu,\quad\text{in }L^2(0,\widetilde{T};H^1(\Omega))\text{ for all $\widetilde{T}\in (0,T_{\rm cr})$};
				\end{align*}
				\item[(V)] there exists a Radon measure $V$ on $(0,T_{\rm cr})\times \Omega\times \mc P$ such that for any $\widetilde{T}\in (0,T_{\rm cr})$ the quadruple $(Q^C, \sigma, \mu,V)$ is a varifold solution to \eqref{eq:sharpinterfacemodel}--\eqref{eq:problemH} in $[0,\widetilde{T}]$ with initial data $(Q^C_0,\sigma_0)$, where $\sigma_0$ is a weak limit of $\sigma^\eps_0$, fulfilling in addition:
				\begin{align*}
					\bullet\,& \d\lambda(t,x)=\d\lambda^t(x)\d t;\\
					\bullet\,& \int_0^{\widetilde{T}}\langle\delta V^\tau,Y(\tau)\rangle\d\tau= \int_0^{\widetilde{T}}\int_\Omega\nabla Y(\tau,x):\left[\d\lambda(\tau,x)I-(\d\lambda_{i,j}(\tau,x))_{d\times d}\right],\\
					&\text{for all }Y\in C^1_0([0,\widetilde T]\times \Omega;\R^d).
				\end{align*}
			\end{itemize}	
		\end{thm}
		\begin{rmk}
			Since assumption \eqref{hyp:Htechnical} implies $H(\pm 1)=0$, and since the limit phase $\varphi$ takes value in $\{+1,-1\}$, we observe that the limit evolution of Problem H solves a decoupled system, meaning that $(Q^C, \mu,V)$ is a varifold solution to the Mullins-Sekerka flow \cite{Chen}, while $\sigma$ is a weak solution to the homogeneous reaction-diffusion equation.
		\end{rmk}
	
	\section{Sharp interface limit for Problem P}\label{sec:sharpinterfaceP}
	
	This section is devoted to the proof of Theorem~\ref{thm:main}, so we tacitly assume all its assumptions. 
	
	In Section~\ref{subsec:boundP} we first provide simple uniform estimates which immediately derive from the energy balance \eqref{eq:EB}. Stronger bounds are then needed in order to overcome the nonlinearities given by $F$ and $P$; we prove them in Lemmas~\ref{lemma:boundw} and \ref{lemma:AL}. A crucial step in the analysis consists in bounding the chemical potential $\mu^\eps$, since only its gradient can be directly controlled from the energy balance. This is done by using an argument by \cite{Chen}, which works (reasonably) as long as the phase variable $\varphi^\eps$ is bounded away from the pure phases $\varphi^\eps\equiv\pm 1$. Differently from the uncoupled Cahn-Hilliard equation \cite{Chen}, where the mass $[\varphi^\eps]$ is automatically conserved through the evolution, in Lemma~\ref{lemma:boundphi} we provide sufficient conditions which ensure a suitable control on the mass of the phase variable in our context.
	
	Such bounds are then used in Section~\ref{subsec:compP} to extract convergent subsequences for the phase variable, the nutrient and the chemical potential. In Section~\ref{subsec:proofP} we finally show that the limit functions actually form a varifold solution to problem \eqref{eq:sharpinterfacemodel}.
	
	\subsection{Uniform bounds}\label{subsec:boundP}
	
	As an immediate corollary of the energy balance \eqref{eq:EB} and properties \eqref{eq:propF1}, \eqref{eq:propF2} we obtain the first uniform bounds, which we keep explicit for future purposes (see Lemma~\ref{lemma:boundphi}):
	
	\begin{prop}\label{prop:unifboundP}
		The following estimates hold  true:
		\begin{itemize}
			\item[(i)] $\sup\limits_{t\in [0,T]}E^\eps(t)\le \mc E_0$;
			\item[(ii)] $\sup\limits_{t\in [0,T]}\||\varphi^\eps(t)|-1\|_{2}\le \sqrt{\mc E_0 \overline c_F^{-1}\eps}$;
			\item[(iii)] $\displaystyle\|\varphi^\eps\|_{B([0,T];L^p(\Omega))}\le \left(\frac{\mc E_0\eps+C_F|\Omega|}{c_F}\right)^\frac 1p$;
			\item[(iv)] $\|\sigma^\eps\|_{B([0,T];L^2(\Omega))}+\|\nabla\sigma^\eps\|_{L^2(0,T;L^2(\Omega))}\le 2\sqrt{\mc E_0}$;
			\item[(v)] $\|\nabla\mu^\eps\|_{L^2(0,T;L^2(\Omega))}\le \sqrt{\mc E_0}$;
			\item[(vi)] $\|\sqrt{P(\varphi^\eps)}(\sigma^\eps-\mu^\eps)\|_{L^2(0,T;L^2(\Omega))}\le \sqrt{\mc E_0}$.
		\end{itemize}
	\end{prop}
	
	In order to gain compactness for the phase variable $\varphi^\eps$ we argue as in \cite{Chen}, introducing the auxiliary functions
	\begin{equation}\label{eq:w}
		w^\eps(t,x):=W(\varphi^\eps(t,x)),
	\end{equation}
	where 
	\begin{equation}\label{eq:W}
		W(u):=\int_{-1}^u\sqrt{2\widetilde{F}(r)}\d r,
	\end{equation}
	and $\widetilde{F}$ is defined as $\widetilde{F}(u):=F(u)\wedge(\max\limits_{[-1,1]}F+u^2)$. We recall (see \cite{Chen} or \cite{MelRoc}) that the function $W$ satisfies
	\begin{equation}\label{eq:propW}
		c|u-v|^2\le |W(u)-W(v)|\le C|u-v|(1+|u|+|v|),\quad\text{for all }u,v\in\R,
	\end{equation}
	for some positive constants $C,c>0$.
	
	\begin{lemma}\label{lemma:boundw}
		The following uniform bounds for the functions $w^\eps$ and $\varphi^\eps$ hold:
		\begin{align}
			&\bullet\, \|\nabla w^\eps(t)\|_{1}\le E^\eps(t)\le \mc E_0,\quad\text{ for all }t\in [0,T];\label{eq:boundw}\\
			&\bullet \, \|w^\eps\|_{B([0,T];W^{1,1}(\Omega))}+\|w^\eps\|_{C^\frac{1}{16}([0,T];L^1(\Omega))}+\|\varphi^\eps\|_{C^\frac{1}{16}([0,T];L^2(\Omega))}\le C.\label{eq:boundholder}
		\end{align}
	\end{lemma}
	\begin{proof}
		We only sketch the proof, being similar to \cite[Lemma 3.2]{Chen}.
		
		Estimate \eqref{eq:boundw} directly follows from the pointwise inequality
		\begin{equation}\label{eq:pointw}
			|\nabla w^\eps(t,x)|=\sqrt{2\widetilde{F}(\varphi^\eps(t,x))}|\nabla\varphi^\eps(t,x)|\le e^\eps(\varphi^\eps(t,x)).
		\end{equation}
		As regards \eqref{eq:boundholder} we first observe that due to \eqref{eq:boundw} it is enough to show
		\begin{equation}\label{eq:thesis}
			\|w^\eps\|_{C^\frac{1}{16}([0,T];L^1(\Omega))}+\|\varphi^\eps\|_{C^\frac{1}{16}([0,T];L^2(\Omega))}\le C.
		\end{equation}
		
		So we fix $\eta_0\in (0,1)$ small enough, and for $\eta\in (0,\eta_0)$ we consider the function
		\begin{equation}\label{eq:mollified}
			\varphi^\eps_\eta(t,x):=\int_{B_1}\rho(y)\varphi^\eps(t,x-\eta y)\d y,
		\end{equation}
		where $B_1$ denotes the unit ball, $\rho$ is a standard mollifier supported in $B_1$, and $\varphi^\eps$ has been properly extended outside $\Omega$. By standard properties of convolution, and exploiting Proposition~\ref{prop:unifboundP} and \eqref{eq:boundw}, for every $t\in[0,T]$ one can show (see (3.4) in \cite{Chen} for details)
		\begin{subequations}
			\begin{align}
				&\bullet\, \|\varphi^\eps_\eta(t)\|_p\le C\|\varphi^\eps(t)\|_p\le C;\label{eq:etaa}\\
				&\bullet \|\nabla \varphi^\eps_\eta(t)\|_2\le \frac C\eta \|\varphi^\eps(t)\|_2\le \frac C\eta;\label{eq:etab}\\
				&\bullet \|\varphi^\eps_\eta(t)-\varphi^\eps(t)\|_2\le C\sqrt{\eta}\sqrt{\|\nabla w^\eps(t)\|_1}\le C\sqrt{\eta}.\label{eq:etac}
			\end{align}
		\end{subequations}
		
		We now fix $0\le s\le t\le T$ and using the equation \eqref{eq:weakeq}$_1$ together with Proposition~\ref{prop:unifboundP} and \eqref{eq:etab}  we estimate the quantity
		\begin{equation}\label{eq:boundJ}
			\begin{aligned}
				J^\eps_\eta(t,s):=&(\varphi^\eps_\eta(t)-\varphi^\eps_\eta(s),\varphi^\eps(t)-\varphi^\eps(s))_2=\int_s^t(\dot\varphi^\eps(\tau),\varphi^\eps_\eta(t)-\varphi^\eps_\eta(s))_2\d \tau\\
				=&{-}\int_{s}^{t}(\nabla\mu^\eps(\tau),\nabla\varphi^\eps_\eta(t){-}\nabla\varphi^\eps_\eta(s))_2\d\tau+\int_s^t(P(\varphi^\eps(\tau))(\sigma^\eps(\tau){-}\mu^\eps(\tau)),\varphi^\eps_\eta(t)-\varphi^\eps_\eta(s))_2\d\tau\\
				\le & \|\nabla\mu^\eps\|_{L^2(0,T;L^2(\Omega))}(\|\nabla\varphi^\eps_\eta(t)\|_2+\|\nabla\varphi^\eps_\eta(s)\|_2)(t-s)^\frac 12\\
				&+ \|\sqrt{P(\varphi^\eps)}(\sigma^\eps-\mu^\eps)\|_{L^2(0,T;L^2(\Omega))}\left(\int_s^t\int_\Omega P(\varphi^\eps(\tau))(\varphi^\eps_\eta(t)-\varphi^\eps_\eta(s))^2\d x\d\tau\right)^\frac 12\\
				\le & \frac C\eta (t-s)^\frac 12+C\left(\underbrace{\int_s^t\int_\Omega P(\varphi^\eps(\tau))(\varphi^\eps_\eta(t)-\varphi^\eps_\eta(s))^2\d x\d\tau}_{I^\eps_\eta(t,s)}\right)^\frac 12.
			\end{aligned}
		\end{equation}
		By H\"older inequality and the growth condition \eqref{eq:propP1} we deduce
		\begin{align*}
			I^\eps_\eta(t,s)\le& \int_s^t\|\varphi^\eps_\eta(t)-\varphi^\eps_\eta(s)\|_p^2\|P(\varphi^\eps(\tau))\|_{\frac{p}{p-2}}\d\tau\\
			\le& C \|\varphi^\eps_\eta(t)-\varphi^\eps_\eta(s)\|_p^2\int_s^t(1+\|\varphi^\eps(\tau)\|^r_{\frac{rp}{p-2}})\d\tau.
		\end{align*}
		Recalling that $rp/(p-2)\le p$ since $r\le p-2$ by assumption, by combining the above two chains of inequalities and using \eqref{eq:etaa} we obtain
		\begin{equation*}
			J^\eps_\eta(t,s)\le C\left(1+\frac 1\eta\right)(t-s)^\frac 12.
		\end{equation*}
		By using triangular inequality and exploiting \eqref{eq:etac} we now infer
		\begin{align*}
			\|\varphi^\eps(t)-\varphi^\eps(s)\|_2^2\le& C\left(1+\frac 1\eta\right)(t{-}s)^\frac 12+\|\varphi^\eps_\eta(t){-}\varphi^\eps(t)\|_2\vee \|\varphi^\eps_\eta(s){-}\varphi^\eps(s)\|_2(	\|\varphi^\eps(t)\|_2{+}\|\varphi^\eps(s)\|_2)\\
			\le & C\left(1+\frac 1\eta\right)(t-s)^\frac 12+C\sqrt{\eta}.
		\end{align*}
		By the arbitrariness of $\eta$, this easily yields $\|\varphi^\eps(t)-\varphi^\eps(s)\|_2\le C(t-s)^\frac{1}{16}$, whence 
		\begin{equation*}
			\|\varphi^\eps\|_{C^\frac{1}{16}([0,T];L^2(\Omega))}\le C.
		\end{equation*}
		
		Let us now focus on $w^\eps$. By \eqref{eq:propW} we first observe that for all $t\in[0,T]$ one has
		\begin{align*}
			\|w^\eps(t)\|_1&=\int_\Omega|W(\varphi^\eps(t))-W(-1)|\d x\le C\int_\Omega|\varphi^\eps(t)+1|(1+|\varphi^\eps(t)|)\d x\\
			&\le C(1+ \|\varphi^\eps(t)\|_2^2)\le C.
		\end{align*}
		Using again \eqref{eq:propW}, for $0\le s\le t\le T$ we then deduce
		\begin{align*}
			\|w^\eps(t)-w^\eps(s)\|_1\le& C\int_\Omega|\varphi^\eps(t)-\varphi^\eps(s)|(1+|\varphi^\eps(t)|+|\varphi^\eps(s)|)\d x\\
			\le &C \|\varphi^\eps(t)-\varphi^\eps(s)\|_2(1+\|\varphi^\eps(t)\|_2+\|\varphi^\eps(s)\|_2)\le C(t-s)^\frac{1}{16}.
		\end{align*}
		Hence \eqref{eq:thesis} is proved and we conclude.
	\end{proof}
	
	As regards the nutrient $\sigma^\eps$, notice that $(iv)$ in Proposition~\ref{prop:unifboundP} already provides weak compactness in $L^2(0,T;H^1(\Omega))$. However, for future purposes, we need to improve such compactness. To this aim, we exploit the following lemma.
	
	\begin{lemma}\label{lemma:AL}
		For any $\alpha>3$, the derivative $\dot\sigma^\eps$ is uniformly bounded in $L^2(0,T;(W^{1,\alpha}(\Omega))^*)$.
	\end{lemma}
	
	\begin{proof}
		Since $\alpha>3$, by Sobolev embedding we observe that $W^{1,\alpha}(\Omega)$ continuously embeds in $C(\overline{\Omega})$. 
		
		We now fix $\Psi\in W^{1,\alpha}(\Omega)$ and for almost every $t\in (0,T)$ we start estimating by using the equation \eqref{eq:weakeq}$_2$:
		\begin{align}\label{eq:1}
			|\langle\dot\sigma^\eps(t),\Psi\rangle_{W^{1,\alpha}(\Omega)}|\le \|\nabla\sigma^\eps(t)\|_2\|\nabla \Psi\|_2+\|\sqrt{P(\varphi^\eps(t))}(\sigma^\eps(t){-}\mu^\eps(t))\|_{2}\|\sqrt{P(\varphi^\eps(t))}\Psi\|_{2}.
		\end{align}
		Note that by \eqref{eq:propP1} there holds
		\begin{equation}\label{eq:2}
			\|\sqrt{P(\varphi^\eps(t))}\Psi\|_{2}\le C(1+\|\varphi^\eps(t)\|_r^\frac r2)\|\Psi\|_{C(\overline{\Omega})}\le C\|\Psi\|_{W^{1,\alpha}(\Omega)},
		\end{equation}
		where in the last inequality we exploited the assumption $r\le p-2\le p$ and $(iii)$ in Proposition~\ref{prop:unifboundP}.
		
		Since $\alpha>3$, estimates \eqref{eq:1} and \eqref{eq:2} finally yield
		\begin{equation*}
			|\langle\dot\sigma^\eps(t),\Psi\rangle_{W^{1,\alpha}(\Omega)}|\le C(\|\nabla\sigma^\eps(t)\|_2+\|\sqrt{P(\varphi^\eps(t))}(\sigma^\eps(t){-}\mu^\eps(t))\|_{2})\|\Psi\|_{W^{1,\alpha}(\Omega)},
		\end{equation*}
		and so we conclude by recalling $(iv)$ and $(v)$ in Proposition~\ref{prop:unifboundP}.
	\end{proof}
	
	In order to bound the chemical potential $\mu^\eps$ in $L^2(\Omega)$, by Poincar\'e-Wirtinger it is enough to estimate its average $[\mu^\eps]$. Arguing as in \cite[Lemma 3.4]{Chen}, this can be done whenever the mass of the phase $\varphi^\eps$ is bounded away from $1$ or $-1$, namely the pure phases.
	
	\begin{lemma}\label{lemma:boundmu}
		Assume that 
		\begin{equation}\label{eq:boundmass}
			|[\varphi^\eps(t)]|\le m_0<1,\qquad\text{for all }t\in [0,T].
		\end{equation}
		Then there exist $\eps_0\in (0,1)$ and a constant $C=C(m_0)>0$, which blows up as $m_0$ goes to $1$, such that
		\begin{equation}\label{eq:boundmu}
			|[\mu^\eps(t)]|\le C(E^\eps(t)+\|\nabla \mu^\eps(t)\|_2),\quad\text{for almost every }t\in (0,T)\text{ and for every }\eps\in (0,\eps_0).
		\end{equation}
		In particular, $\mu^\eps$ is uniformly bounded in $L^2(0,T;H^1(\Omega))$.
	\end{lemma} 
	\begin{proof}
		Although the proof follows the lines of \cite[Lemma 3.4]{Chen}, we write it both for the sake of completeness and since some formulas will be used in the sequel. Notice that classic arguments like Miranville-Zelik inequality \cite{MirZelik}, which exploits an assumption of the form \eqref{hyp:F3}$_1$ (see also \cite{GilMirSchimp}), do not work here due to the presence of the constant $1/\eps$ in front of the derivative $F'$.
		
		Let $Y\in C^1(\overline{\Omega};\R^d)$ be a smooth vector field. We first observe that, exploiting the equation 
			\begin{equation*}
				\mu^\eps(t)=-\eps\Delta\varphi^\eps(t)+\frac 1\eps F'(\varphi^\eps(t)),\qquad \text{in }\Omega,
			\end{equation*}
			after integration by parts we deduce
			\begin{align*}
				{-}\eps\int_\Omega\nabla Y:(\nabla\varphi^\eps(t)\otimes\nabla\varphi^\eps(t))\d x&=\eps\int_\Omega Y\cdot \left(\nabla \left(\frac 12 |\nabla\varphi^\eps(t)|^2\right)+\nabla\varphi^\eps(t)\Delta \varphi^\eps(t)\right)\d x\\
				&=\int_\Omega Y\cdot \left(\nabla \Big(e^\eps(\varphi^\eps(t)\Big)-\mu^\eps(t)\nabla \varphi^\eps(t)\right)\d x,
			\end{align*}
			whence
			\begin{align*}
				&\qquad\int_\Omega\nabla Y:(e^\eps(\varphi^\eps(t))I-\eps \nabla\varphi^\eps(t)\otimes\nabla\varphi^\eps(t))\d x\\
				&=\int_\Omega \Big(e^\eps(\varphi^\eps(t))\div Y+Y\cdot \big(\nabla \Big(e^\eps(\varphi^\eps(t)\Big)-\mu^\eps(t)\nabla \varphi^\eps(t)\big)\Big) \d x\\
				&= \int_\Omega\div \Big(\big(e^\eps(\varphi^\eps(t))-\varphi^\eps(t)\mu^\eps(t)\big)Y\Big)\d x+\int_\Omega (\varphi^\eps(t)\mu^\eps(t)\div Y+ \varphi^\eps(t)\nabla \mu^\eps(t)\cdot Y) \d x.
			\end{align*}
			After a further integration by parts, and by summing and subtracting the terms containing $[\mu^\eps(t)]$, we then obtain
			\begin{equation}\label{eq:tensor}
				\begin{aligned}
					&\int_\Omega\nabla Y:(e^\eps(\varphi^\eps(t))I-\eps \nabla\varphi^\eps(t)\otimes\nabla\varphi^\eps(t))\d x-\int_{\partial\Omega}e^\eps(\varphi^\eps(t))Y\cdot n\d\mc H^{d-1}\\
					=& -\int_{\partial\Omega}\varphi^\eps(t)\mu^\eps(t)Y\cdot n\d\mc H^{d-1}+[\mu^\eps(t)]\int_\Omega\varphi^\eps(t)\div Y\d x+\int_\Omega Y\cdot\nabla\mu^\eps(t)\varphi^\eps(t)\d x\\
					&+\int_\Omega(\mu^\eps(t)-[\mu^\eps(t)])\varphi^\eps(t)\div Y\d x.
				\end{aligned}
			\end{equation}
		
		By choosing $Y=\nabla \psi$, for an arbitrary $\psi\in C^2(\overline{\Omega})$ with $\partial_n\psi=0$ on $\partial\Omega$, we finally deduce
		\begin{equation}\label{eq:avmu}
			\begin{aligned}
				[\mu^\eps(t)]=\frac{1}{\int_\Omega\Delta\psi\varphi^\eps(t)\d x}\Big(\int_\Omega\nabla^2\psi&:(e^\eps(\varphi^\eps(t))I-\eps\nabla\varphi^\eps(t)\otimes\nabla\varphi^\eps(t))\\
				&-\varphi^\eps(t)\nabla\mu^\eps(t)\cdot\nabla\psi-\varphi^\eps(t)(\mu^\eps(t)-[\mu^\eps(t)])\Delta\psi\d x\Big).
			\end{aligned}
		\end{equation}
		Fix then $\eta>0$ and consider the function $\varphi^\eps_\eta$ defined in \eqref{eq:mollified}. Let now $\psi^\eps_\eta(t)$ be the unique solution of the problem
		\begin{equation*}
			\begin{cases}
				\Delta\psi=\varphi^\eps_\eta(t)-[\varphi^\eps_\eta(t)],&\text{in }\Omega,\\
				\partial_n\psi=0,&\text{in }\partial\Omega,\\
				[\psi]=0.		
			\end{cases}
		\end{equation*}
		By elliptic regularity one deduces $\psi^\eps_\eta(t)\in C^2(\overline{\Omega})$ with
		\begin{equation*}
			\|\psi^\eps_\eta(t)\|_{C^2(\overline{\Omega})}\le C \|\varphi^\eps_\eta(t)\|_{C^1(\overline{\Omega})}\le \frac{C}{\eta^\beta},
		\end{equation*}
		for some $\beta>0$. The second inequality above follows by properties of convolutions, since there holds
		\begin{equation*}
			\|\varphi^\eps_\eta(t)\|_{C^1(\overline{\Omega})}\le \frac{C}{\eta^\beta}	\|\varphi^\eps(t)\|_{2}\le \frac{C}{\eta^\beta}.
		\end{equation*}
		Thus, by choosing $\psi=\psi^\eps_\eta(t)$, the numerator $N^\eps_\eta(t)$ in \eqref{eq:avmu} can be bounded by
		\begin{equation}\label{eq:N}
			\begin{aligned}
				|N^\eps_\eta(t)|&\le C\|\psi^\eps_\eta(t)\|_{C^2(\overline{\Omega})}\bigg(E^\eps(t)+\|\varphi^\eps(t)\|_{2}\|\nabla\mu^\eps(t)\|_{2}+\|\varphi^\eps(t)\|_{2}\|\mu^\eps(t)-[\mu^\eps(t)]\|_{2}\bigg)\\
				&\le \frac{C}{\eta^\beta}\big(E^\eps(t)+\|\nabla\mu^\eps(t)\|_{2}\big),
			\end{aligned}
		\end{equation}
		where in the last inequality we exploited Poincar\'e-Wirtinger inequality and the uniform bound of $\varphi^\eps$ in $L^2(\Omega)$. On the other hand the denominator $D^\eps_\eta(t)$ in \eqref{eq:avmu} can be written as follows:
		\begin{equation*}
			\begin{aligned}				D^\eps_\eta(t)=&\int_\Omega\Delta\psi^\eps_\eta(t)\varphi^\eps(t)\d x=\int_\Omega (\varphi^\eps_\eta(t)-[\varphi^\eps_\eta(t)]) \varphi^\eps(t)\d x= \int_\Omega(\varphi^\eps_\eta(t)-\varphi^\eps(t))\varphi^\eps(t)\d x\\
				&+\int_\Omega(\varphi^\eps(t)^2-1)\d x+|\Omega|(1-[\varphi^\eps(t)]^2)+|\Omega|[\varphi^\eps(t)]([\varphi^\eps(t)]-[\varphi^\eps_\eta(t)]).
			\end{aligned}
		\end{equation*}
		Observing that
		\begin{align*}
			&\bullet\, |[\varphi^\eps(t)]-[\varphi^\eps_\eta(t)]|\le C\|\varphi^\eps(t)-\varphi^\eps_\eta(t)\|_2\le C\sqrt{\eta};\\
			&\bullet\, |[\varphi^\eps(t)]|\le C\|\varphi^\eps(t)\|_2\le C;\\
			&\bullet\, \int_\Omega|\varphi^\eps(t)^2-1|\d x\le C(1+\|\varphi^\eps(t)\|_2)\||\varphi^\eps(t)|-1\|_2\le C\sqrt{\eps},
		\end{align*}
		by assumption \eqref{eq:boundmass} we now deduce
		\begin{equation}\label{eq:D}
			D^\eps_\eta(t)\ge |\Omega|(1-m_0^2)-C(\sqrt{\eta}+\sqrt{\eps}).
		\end{equation}
		We finally conclude combining \eqref{eq:N} and \eqref{eq:D}, by choosing $\eta$ sufficiently small and setting $\eps_0:=\eta$. By \eqref{eq:D} the constant in \eqref{eq:boundmu} clearly blows up when $m_0\to 1$.
	\end{proof}
	
	We now show that both conditions \eqref{eq:ass1} and \eqref{eq:ass2} ensure the validity of \eqref{eq:boundmass}.
	
	\begin{lemma}\label{lemma:boundphi}
		Assume that either \eqref{eq:ass1} or \eqref{eq:ass2} is in force. Then there exist $m_0\in [0,1)$ and $\eps_1\in (0,1]$ such that \eqref{eq:boundmass} holds true for all $\eps\in (0,\eps_1)$.
	\end{lemma}
	\begin{proof}
		We begin by assuming \eqref{eq:ass1}. Let $\eps_1:=1\wedge \frac{|\Omega|C_F}{\mc E_0}$ and fix $t\in [0,T]$ and $\eps\in (0,\eps_1)$. Observe that by $(iii)$ in Proposition~\ref{prop:unifboundP} we have
		\begin{equation}\label{eq:useful}
			\|\varphi^\eps\|_{B([0,T];L^p(\Omega))}\le \left(\frac{\mc E_0\eps+C_F|\Omega|}{c_F}\right)^\frac 1p\le \left(2|\Omega|\frac{C_F}{c_F}\right)^\frac 1p.
		\end{equation}
		Now we estimate using the equation \eqref{eq:weakeq}$_1$, the energy balance \eqref{eq:EB} and the growth condition \eqref{eq:propP1}, recalling also \eqref{eq:inmassbound}:
		\begin{align*}
			|[\varphi^\eps(t)]|&\le |[\varphi^\eps_0]|+\frac{1}{|\Omega|}\left|\int_0^t(\dot\varphi^\eps(\tau),1)_2\d\tau\right|=|[\varphi^\eps_0]|+\frac{1}{|\Omega|}\left|\int_0^t(P(\varphi^\eps(\tau))(\sigma^\eps(\tau){-}\mu^\eps(\tau)),1)_2\d\tau\right|\\
			&\le c_0+\frac{1}{|\Omega|}\|\sqrt{P(\varphi^\eps)}(\sigma^\eps{-}\mu^\eps)\|_{L^2(0,T;L^2(\Omega))}\left(\int_0^T\int_\Omega P(\varphi^\eps(\tau))\d x\d\tau\right)^\frac 12\\
			&\le c_0+\left(\frac{\mc E_0 C_P}{|\Omega|^2}\int_0^T\int_\Omega (1+|\varphi^\eps(\tau)|^r)\d x\d\tau\right)^\frac 12.
		\end{align*}
		Since $r\le p-2$ by assumption, one has $|u|^r\le 1+|u|^p$ for all $u\in \R$, hence by \eqref{eq:useful} we infer
		\begin{align*}
			|[\varphi^\eps(t)]|&\le c_0+ \left[\frac{\mc E_0 C_P}{|\Omega|^2}\left(2T|\Omega|+\int_0^T \|\varphi^\eps(\tau)\|_p^p\d\tau\right)\right]^\frac 12\\
			&\le c_0+\left(2T\mc E_0 C_P\frac{c_F+C_F}{|\Omega|c_F}\right)^\frac 12=:m_0.
		\end{align*}
		By \eqref{eq:ass1} one has $m_0<1$ and so we conclude the first part of the proof.
		
		If \eqref{eq:ass2} is instead in force, we first set $\eps_1:=1$. By means of \eqref{eq:massconservation} and the energy balance \eqref{eq:EB}, for all $t\in [0,T]$ and $\eps\in (0,\eps_1)$ we then have
		\begin{align*}
			|[\varphi^\eps(t)]|&= |[\varphi^\eps_0+\sigma^\eps_0]-[\sigma^\eps(t)]|\le d_0+|[\sigma^\eps(t)]|\le d_0+ \frac{1}{|\Omega|^\frac12}\|\sigma^\eps(t)\|_2\\
			&\le d_0+\left(2\frac{\mc E_0}{|\Omega|}\right)^\frac 12=:m_0.
		\end{align*}
		Again by \eqref{eq:ass2} one has $m_0<1$ and so the proof is complete.
	\end{proof}
	
	We conclude the section with the following a priori estimate on the (positive part of the) discrepancy density
	\begin{equation}\label{eq:discrepancy}
		\xi^\eps(\varphi^\eps(t)):=\frac\eps 2|\varphi^\eps(t)|^2-\frac 1\eps F(\varphi^\eps(t)),
	\end{equation}
	whose proof can be found in \cite[Theorem~3.6]{Chen}.
	
	\begin{lemma}
		There exists $\eta_0\in (0,1]$, and there exist two continuous and nonincreasing functions $M_1,M_2\colon (0,\eta_0]\to (0,+\infty)$ such that the following holds. For every $\eta\in (0,\eta_0)$ and for every $\eps\in \left(0,\frac{1}{M_1(\eta)}\right)$ one has
		\begin{equation*}
			\int_0^T\int_\Omega \xi^\eps(\varphi^\eps(\tau))^+\d x\d\tau\le \eta\int_0^TE^\eps(\tau)\d\tau+\eps M_2(\eta)\|\mu^\eps\|^2_{L^2(0,T;L^2(\Omega))}.
		\end{equation*}
		In particular, whenever Lemma~\ref{lemma:boundmu} applies, there holds
		\begin{equation}\label{eq:discvanish}
			\lim\limits_{\eps\to 0}\int_0^T\int_\Omega \xi^\eps(\varphi^\eps(\tau))^+\d x\d\tau=0.
		\end{equation}
	\end{lemma}
	\subsection{Compactness}\label{subsec:compP}
	The uniform bounds previously obtained allow one to infer the following compactness result.
	\begin{prop}\label{prop:compactnessP}
		There exist a (non relabelled) subsequence $\eps\to 0$, a bounded measurable function $E\colon [0,T]\to [0,+\infty)$, a measurable set $Q^C\subseteq Q$, a function $\sigma\in L^2(0,T;H^1(\Omega))$ and Radon measures $\lambda, \lambda_{i,j}$ on $\overline{Q}$ for $i,j=1,\dots,d$ such that the following facts hold true:
		\begin{enumerate}
			\item $E^\eps(t)\xrightarrow[\eps\to 0]{}E(t)$ for almost every $t\in (0,T)$;
			\item $w^\eps\xrightarrow[\eps\to 0]{}2\theta \chi_{Q^C}$ in $C([0,T];L^1(\Omega))$, and $\chi_{Q^C}$ is in $C^\frac 18([0,T];L^1(\Omega))\cap B([0,T];BV(\Omega))$;
			\item $\varphi^\eps\xrightarrow[\eps\to 0]{}-1+2\chi_{Q^C}=:\varphi$ in $C([0,T];L^2(\Omega))$, and $\varphi^\eps(t)\xrightarrow[\eps\to 0]{}\varphi(t)$ in $L^q(\Omega)$ for all $q\in [1,p)$ and $t\in [0,T]$;
			\item $\sigma^\eps \xrightharpoonup[\eps\to 0]{}\sigma$ in $L^2(0,T;H^1(\Omega))$ and $\sigma^\eps \xrightarrow[\eps\to 0]{}\sigma$ in $L^2(0,T;L^q(\Omega))$ for all $q\in [1,6)$;
			\item $e^\eps(\varphi^\eps)\d t\d x\xrightharpoonup[\eps\to 0]{\ast}\lambda$ and $\eps\partial_i\varphi^\eps\partial_j\varphi^\eps\d t\d x\xrightharpoonup[\eps\to 0]{\ast}\lambda_{i,j}$ in the sense of measures in $\overline Q$.
		\end{enumerate}
		If in addition \eqref{eq:ass1} or \eqref{eq:ass2} is in force, then there exists a function $\mu\in L^2(0,T;H^1(\Omega))$ such that
		\begin{enumerate}[start=6]
			\item $\mu^\eps \xrightharpoonup[\eps\to 0]{}\mu$ in $L^2(0,T;H^1(\Omega))$;
			\item $\sqrt{P(\varphi^\eps)}(\sigma^\eps{-}\mu^\eps)\xrightharpoonup[\eps\to 0]{}\sqrt{P(\varphi)}(\sigma{-}\mu)$ in $L^2(0,T;L^2(\Omega))$.
		\end{enumerate}
	\end{prop}
	\begin{proof}
		The validity of $(5)$ and of the weak convergence in $(4)$ directly follows from the uniform bounds of $e^\eps(\varphi^\eps)$ and $\sigma^\eps$ in $L^1(Q)$ and $L^2(0,T;H^1(\Omega))$, respectively. The strong convergence in $(4)$ instead can be obtained by the Aubin-Lions compactness theorem: indeed, by Lemma~\ref{lemma:AL} we know that $\sigma^\eps$ is also bounded in $H^1(0,T;(W^{1,\alpha}(\Omega))^*)$ if $\alpha>3$, and for $q\in [2,6)$ the inclusions $H^1(\Omega)\hookrightarrow\hookrightarrow L^q(\Omega)\hookrightarrow(W^{1,\alpha}(\Omega))^*$ hold.
		
		We now show $(1)$. From the energy balance \eqref{eq:EB} we deduce that the functions $t\mapsto E^\eps(t)+\frac 12 \|\sigma^\eps(t)\|_2^2$ are nondecreasing and uniformly bounded; thus, by Helly's Selection Theorem they pointwise converge to a nondecreasing function in $[0,T]$, call it $\widetilde E$. On the other hand, for almost every $t\in(0,T)$ we have $\sigma^\eps(t)\xrightarrow[\eps\to 0]{}\sigma(t)$ in $L^2(\Omega)$. Hence for almost every $t\in (0,T)$ we infer $E^\eps(t)\xrightarrow[\eps\to 0]{}\widetilde{E}(t)-\frac 12 \|\sigma(t)\|_2^2=:E(t)$.
		
		As regards $(2)$ and $(3)$ we first observe that, in view of \eqref{eq:boundholder}, an application of Ascoli-Arzel\'a Theorem yields $w^\eps\xrightarrow[\eps\to 0]{}w$ in $C([0,T];L^1(\Omega))$. After defining $\varphi(t,x):=W^{-1}(w(t,x))$ (notice that $W$ is invertible by definition \eqref{eq:W}), by using \eqref{eq:w} and \eqref{eq:propW} it is then easy to deduce $\varphi^\eps\xrightarrow[\eps\to 0]{}\varphi$ in $C([0,T];L^2(\Omega))$. So, as a consequence of $(ii)$ in Proposition~\ref{prop:unifboundP}, it must be $|\varphi|=1$ almost everywhere in $Q$, namely $\varphi$ has the form $\varphi=-1+2\chi_{Q^C}$ for some measurable set $Q_C\subseteq Q$. Recalling \eqref{eq:theta}, this finally implies
		\begin{equation*}
			w(t,x)=W(\varphi(t,x))=\int_{-1}^{\varphi(t,x)}\sqrt{2\widetilde{F}(u)}\d u=\int_{-1}^{-1+2\chi_{Q^C}(t,x)}\sqrt{2{F}(u)}\d u=2\theta\chi_{Q^C}(t,x).
		\end{equation*}
		Moreover, as a consequence of \eqref{eq:boundholder}, observe that for $0\le s\le t\le T$ one has
		\begin{equation*}
			\|\chi_{Q^C_t}-\chi_{Q^C_s}\|_1=\lim\limits_{\eps\to 0}\frac 14\|\varphi^\eps(t)-\varphi^\eps(s)\|_2^2\le C(t-s)^\frac 18.
		\end{equation*}
		By weak compactness in $BV(\Omega)$, from \eqref{eq:boundholder} one also obtains $\chi_{Q^C}\in B([0,T];BV(\Omega))$. Furthermore, since $\varphi^\eps(t)\xrightarrow[\eps\to 0]{}\varphi(t)$ in $L^2(\Omega)$ and $\varphi^\eps(t)$ is bounded in $L^p(\Omega)$, a standard argument with Lebesgue spaces yields $\varphi^\eps(t)\xrightarrow[\eps\to 0]{}\varphi(t)$ in $L^q(\Omega)$ for $q\in [1,p)$.
		
		Assume now that \eqref{eq:ass1} or \eqref{eq:ass2} is in force. Lemmas~\ref{lemma:boundmu} and \ref{lemma:boundphi} then immediately yield $(6)$. 
		
		Furthermore, from $(vi)$ in Proposition~\ref{prop:unifboundP} we deduce that 
		\begin{equation*}
			\sqrt{P(\varphi^\eps)}(\sigma^\eps{-}\mu^\eps)\xrightharpoonup[\eps\to 0]{}g,\qquad\text{in }L^2(0,T;L^2(\Omega)).
		\end{equation*}
		We just need to show $g=\sqrt{P(\varphi)}(\sigma{-}\mu)$ in order to prove $(7)$. By uniqueness of the limit, it is enough to prove that
		\begin{equation*}
			\sqrt{P(\varphi^\eps)}(\sigma^\eps{-}\mu^\eps)\xrightharpoonup[\eps\to 0]{}\sqrt{P(\varphi)}(\sigma{-}\mu),\qquad\text{in }L^2(0,T;L^\delta(\Omega))\quad\text{ for some $\delta\in (1,2)$},
		\end{equation*}
		which in turn is a byproduct of
		\begin{equation}\label{eq:deltaprime}
			\sqrt{P(\varphi^\eps)}\psi\xrightarrow[\eps\to 0]{}	\sqrt{P(\varphi)}\psi\quad\text{in }L^2(0,T;L^\frac 65(\Omega)),\quad\text{ for all }\psi\in L^2(0,T;L^{\delta'}(\Omega))\quad\text{ for some $\delta'>2$},
		\end{equation}
		observing that $\sigma^\eps-\mu^\eps\xrightharpoonup[\eps\to 0]{}\sigma-\mu$ in $L^2(0,T;L^6(\Omega))$.
		
		In order to show $\eqref{eq:deltaprime}$ we first notice that since $r\le p-2$ we also have $p>6/7(r-1)$, and so there exists $\delta'>2$ for which
		\begin{equation}\label{eq:condition}
			\frac{6\delta'}{7\delta'-12}(r-1)\le p.
		\end{equation}
		We then estimate by using H\"older inequality:
		\begin{align*}
			&\|\sqrt{P(\varphi^\eps)}\psi-	\sqrt{P(\varphi)}\psi\|^2_{L^2(0,T;L^\frac 65(\Omega))}\le \int_0^T\|\psi(\tau)\|^2_{\delta'}\|{P(\varphi^\eps(\tau))}-	{P(\varphi(\tau))}\|_{\frac{3\delta'}{5\delta'-6}}\d\tau\\
			\le&\|\psi\|_{L^2(0,T;L^{\delta'}(\Omega))}^2\sup\limits_{t\in [0,T]}\|{P(\varphi^\eps(t))}-	{P(\varphi(t))}\|_{\frac{3\delta'}{5\delta'-6}}.
		\end{align*}
		We now observe that for a fixed $t\in [0,T]$ by means of \eqref{eq:propP2} and \eqref{eq:condition} we have (without loss of generality here we assume $r>1$)
		\begin{align*}
			\|{P(\varphi^\eps(t))}-	{P(\varphi(t))}\|_{\frac{3\delta'}{5\delta'-6}}&\le C\|\varphi^\eps(t)-\varphi(t)\|_2\left(1+\|\varphi^\eps(t)\|^{r-1}_{\frac{6\delta'}{7\delta'-12}(r-1)}\right)\\
			&\le C\|\varphi^\eps(t)-\varphi(t)\|_2\left(1+\|\varphi^\eps(t)\|^{r-1}_{p}\right)\le C\|\varphi^\eps(t)-\varphi(t)\|_2,
		\end{align*}
		which vanishes uniformly. So \eqref{eq:deltaprime} is proved and we conclude.
	\end{proof}
	Finally, the following lemma will be used to pass to the limit the right-hand side of equation \eqref{eq:weakeq}.
	\begin{lemma}\label{lemma:convP}
		Consider the subsequence and the limit objects given by Proposition~\ref{prop:compactnessP}. Then for any $\Phi\in C([0,T]\times\overline{\Omega})$ and for all $0\le s\le t\le T$ there holds
		\begin{equation*}
			\lim\limits_{\eps\to 0}\int_s^t(P(\varphi^\eps(\tau))(\sigma^\eps(\tau){-}\mu^\eps(\tau)),\Phi(\tau))_2\d\tau=\int_s^t(P(\varphi(\tau))(\sigma(\tau){-}\mu(\tau)),\Phi(\tau))_2\d\tau.
		\end{equation*}
	\end{lemma}
	\begin{proof}
		Due to $(7)$ in Proposition~\ref{prop:compactnessP}, it is enough to show
		\begin{equation}\label{eq:enough}
			\sqrt{P(\varphi^\eps)}\Phi\xrightarrow[\eps\to 0]{}\sqrt{P(\varphi)}\Phi,\quad\text{ in }L^2(0,T;L^2(\Omega)).
		\end{equation}
		We thus estimate by using \eqref{eq:propP2} and recalling that $|\varphi|=1$ (without loss of generality here we assume $r>1$):
		\begin{align*}
			&\|(\sqrt{P(\varphi^\eps)}-\sqrt{P(\varphi)})\Phi\|_{L^2(0,T;L^2(\Omega))}\le \|\Phi\|_{C([0,T]\times\overline\Omega)}\|P(\varphi^\eps)-P(\varphi)\|^\frac 12_{L^1(0,T;L^1(\Omega))}\\
			\le &C\left(\int_0^T\int_\Omega|\varphi^\eps(\tau)-\varphi(\tau)|(1+|\varphi^\eps(\tau)|^{r-1})\d x\d\tau\right)^\frac 12\\
			&\le C\left(\int_0^T\|\varphi^\eps(\tau)-\varphi(\tau)\|_q(1+\|\varphi^\eps(\tau)\|^{r-1}_{q\frac{r-1}{q-1}})\d\tau\right)^\frac 12\\
			&\le C\|\varphi^\eps-\varphi\|^\frac 12_{L^1(0,T;L^q(\Omega))}\sup\limits_{t\in [0,T]}(1+\|\varphi^\eps(t)\|^\frac{r-1}{2}_{q\frac{r-1}{q-1}}),
		\end{align*}
		where $q\in (1,p)$. By choosing $q\ge p/3$, since $r\le p-2$ by assumption, one has $q\frac{r-1}{q-1}\le p$, and so the term within brackets is uniformly bounded by $(iii)$ in Proposition~\ref{prop:unifboundP}. On the other hand, the term in front is vanishing by Dominated Convergence Theorem due to $(3)$ in Proposition~\ref{prop:compactnessP}. Hence \eqref{eq:enough} is proved and we conclude.
	\end{proof}
	\subsection{Proof of Theorem~\ref{thm:main}}\label{subsec:proofP}
	
	We are finally in a position to prove Theorem~\ref{thm:main}. We first observe that $(I)$, $(II)$ and $(III)$ have been already obtained in Proposition~\ref{prop:compactnessP}.
	
	We then fix $\widetilde{T}\in (0,T_{\rm cr})$, so that by definition \eqref{eq:Tmax} and by the uniform convergence $\varphi^\eps\to\varphi$ in $L^2(\Omega)$ we automatically have \eqref{eq:boundmass} in $[0,\widetilde{T}]$, and so Lemma~\ref{lemma:boundmu} together with a diagonal argument yields the validity of $(IV)$.
	
	In order to prove $(V)$, let us first notice that conditions $(a)$ and $(b)$ in Definition~\ref{def:varifoldsol} are fulfilled (in $[0,\widetilde{T}]$). We now check $(d)$: fix $\Phi,\Psi\in C^1([0,\widetilde{T}]\times\overline \Omega)$ such that $\Phi(\widetilde T)=\Psi(\widetilde T)=0$. After integration by parts (in time) in \eqref{eq:weakeq} we deduce
	\begin{align*}
		\int_0^{\widetilde T}\bigg({-}(1+\varphi^\eps(\tau),\dot\Phi(\tau))_2&+(\nabla\mu^\eps(\tau),\nabla\Phi(\tau))_{2}-(P(\varphi^\eps(\tau))(\sigma^\eps(\tau){-}\mu^\eps(\tau)),\Phi(\tau))_{2}\bigg)\d\tau\\
		&=(1+\varphi^\eps_0,\Phi(0))_2,
	\end{align*}
	and 
	\begin{align*}
		\int_0^{\widetilde T}\bigg({-}(\sigma^\eps(\tau),\dot\Psi(\tau))_{2}&+(\nabla\sigma^\eps(\tau),\nabla\Psi(\tau))_{2}+(P(\varphi^\eps(\tau))(\sigma^\eps(\tau){-}\mu^\eps(\tau))),\Psi(\tau))_{2}\bigg)\d\tau\\
		&=(\sigma^\eps_0,\Psi(0))_{2}.
	\end{align*}
	Letting $\eps\to 0$ and exploiting Proposition~\ref{prop:compactnessP} and Lemma~\ref{lemma:convP} we obtain $(d)$, also using the density of $C^1([0,\widetilde{T}]\times\overline \Omega)$ in $L^2(0,\widetilde T;H^1(\Omega))\cap H^1(0,\widetilde T;L^2(\Omega))$.
	
	We thus need to build the measure $V$ in such a way that also conditions $(c)$, $(e)$ and $(f)$ are satisfied. To this aim we recall that by \eqref{eq:tensor} for any $Y\in C^1_0([0,\widetilde T]\times \Omega;\R^d)$ we have
	
	\begin{equation*}
		\begin{aligned}
			&\int_0^{\widetilde{T}}\int_\Omega\nabla Y(\tau):(e^\eps(\varphi^\eps(\tau))I-\eps \nabla\varphi^\eps(\tau)\otimes\nabla\varphi^\eps(\tau))\d x\d\tau\\
			=& \int_0^{\widetilde{T}}\int_{\Omega}\varphi^\eps(\tau)(\mu^\eps(\tau)\div Y(\tau)+\nabla\mu^\eps(\tau)\cdot Y(\tau))\d x\d\tau.
		\end{aligned}
	\end{equation*}
	
	Letting $\eps\to 0$ we hence obtain the identity
	\begin{equation}\label{eq:tenslimit}
		\begin{aligned}
			&\int_0^{\widetilde{T}}\int_\Omega\nabla Y(\tau):(\d\lambda(\tau,x)I-(\d\lambda_{i,j}(\tau,x))_{d\times d})\\
			=& \int_0^{\widetilde{T}}\int_{\Omega}(-1+2\chi_{Q^C_\tau})(\mu(\tau)\div Y(\tau)+\nabla\mu(\tau)\cdot Y(\tau))\d x\d\tau\\
			=&\int_0^{\widetilde{T}}\int_{\Omega}(-1+2\chi_{Q^C_\tau})\div(\mu(\tau)Y(\tau))\d x\d\tau=2\int_0^{\widetilde{T}}\int_{Q^C_\tau}\div(\mu(\tau)Y(\tau))\d x\d\tau.
		\end{aligned}
	\end{equation}
	
	Since $e^\eps(\varphi^\eps)$ is actually bounded in $L^\infty(0,T;L^1(\Omega))$ by $(i)$ in Proposition \ref{prop:unifboundP}, we now observe that for all $0\le s\le t\le \widetilde{T}$ we have
	\begin{equation*}
		\int_{[s,t]\times\overline{\Omega}}\d\lambda(\tau,x)=\lim\limits_{\eps\to 0}\int_s^t\int_\Omega e^\eps(\varphi^\eps(\tau))\d x\d\tau=\lim\limits_{\eps\to 0}\int_s^tE^\eps(\tau)\d\tau=\int_s^t E(\tau)\d\tau,
	\end{equation*}
	whence the splitting $\d\lambda(t,x)=\d\lambda^t(x)\d t$ for some Radon measure $\lambda^t$ on $\overline{\Omega}$. In particular $\lambda^t(\overline{\Omega})=E(t)$ for almost every $t\in (0,\widetilde T)$, and so for almost every $0\le s\le t\le \widetilde{T}$ by weak lower-semicontinuity and $(1)$ in Proposition~\ref{prop:compactnessP} there holds
	\begin{align*}
		&\lambda^t(\overline{\Omega})+\frac 12 \|\sigma(t)\|^2_{2}+\int_s^t\left(\|\nabla\mu(\tau)\|^2_{2}+\|\nabla\sigma(\tau)\|^2_{2}+\|\sqrt{P(\varphi(\tau))}(\sigma(\tau){-}\mu(\tau))\|^2_{2}\right)\d\tau\\
		\le&\liminf\limits_{\eps\to 0}\left(E^\eps(t){+}\frac 12 \|\sigma^\eps(t)\|^2_{2}{+}\int_s^t\left(\|\nabla\mu^\eps(\tau)\|^2_{2}{+}\|\nabla\sigma^\eps(\tau)\|^2_{2}{+}\|\sqrt{P(\varphi^\eps(\tau))}(\sigma^\eps(\tau){-}\mu^\eps(\tau))\|^2_{2}\right)\d\tau\right)\\	
		=&\lim\limits_{\eps\to 0}\left(E^\eps(s)+\frac 12 \|\sigma^\eps(s)\|^2_{2}\right)=E(s)+\frac 12 \|\sigma(s)\|^2_{2}= \lambda^s(\overline{\Omega})+\frac 12 \|\sigma(s)\|^2_{2},
	\end{align*}
	namely $(f)$ is fulfilled.
	
	We also notice that for $Y,Z\in C([0,\widetilde T]\times\overline \Omega;\R^d)$ there holds
	\begin{align*}
		&\int_0^{\widetilde{T}}\int_\Omega \eps(\nabla\varphi^\eps(\tau)\otimes \nabla\varphi^\eps(\tau))Z(\tau)\cdot Y(\tau) \d x\d\tau\\
		\le & \int_0^{\widetilde{T}}\int_\Omega \eps|\nabla\varphi^\eps(\tau)|^2|Z(\tau)||Y(\tau)| \d x\d\tau\\
		\le &\int_0^{\widetilde{T}}\int_\Omega e^\eps(\varphi^\eps(\tau))|Z(\tau)||Y(\tau)| \d x\d\tau+\|Y\|_{C([0,\widetilde T]\times\overline \Omega)}\|Z\|_{C([0,\widetilde T]\times\overline \Omega)}\int_0^{\widetilde{T}}\int_\Omega \xi^\eps(\varphi^\eps(\tau))^+ \d x\d\tau,
	\end{align*}
	where we used the definition of the energy density \eqref{eq:endens} and of the discrepancy density \eqref{eq:discrepancy}.
	
	Sending $\eps\to 0$ and exploiting \eqref{eq:discvanish} one obtains
	\begin{equation}\label{eq:AC}
		\int_{[0,\widetilde{T}]\times\overline\Omega} (\d\lambda_{i,j}(\tau,x))_{d\times d}Z(\tau,x)\cdot Y(\tau,x) \le \int_{[0,\widetilde{T}]\times\overline\Omega} |Z(\tau,x)||Y(\tau,x)| \d\lambda(\tau,x).
	\end{equation}
	By the arbitrariness of the vector fields $Y,Z$ we thus infer that $\lambda_{i,j}$ is absolutely continuous with respect to $\lambda$ for any $i,j=1,\dots d$, and so the exist $\lambda$-measurable functions $g_{i,j}$ such that $\lambda_{i,j}=g_{i,j}\lambda$. This equality, together with the fact that $\lambda_{i,j}$ are limit of simmetric and positive-definite matrices and \eqref{eq:AC}, implies
	\begin{equation*}
		0\le (g_{i,j})_{d\times d}=(g_{j,i})_{d\times d}\le I,\quad\text{$\lambda$-a.e. in } [0,\widetilde T]\times\overline \Omega,
	\end{equation*}
	whence $(g_{i,j})_{d\times d}=\sum_{i=1}^d\alpha_iv_i\otimes v_i$ for some $\lambda$-measurable functions $\alpha_i$ and unit vectors $v_i$ satisfying
	\begin{equation}\label{eq:propcoeff}
		0\le \alpha_i\le 1,\quad \sum_{i=1}^{d}\alpha_i\le 1,\quad \sum_{i=1}^{d}v_i\otimes v_i=I,\quad \lambda\text{-a.e. in }[0,\widetilde T]\times\overline \Omega.
	\end{equation}
	By exploiting the splitting $\d\lambda(t,x)=\d\lambda^t(x)\d t$, from \eqref{eq:tenslimit} we now deduce
	\begin{equation}\label{eq:imp}
		\begin{aligned}
			2\int_0^{\widetilde{T}}\int_{Q^C_\tau}\div(\mu(\tau)Y(\tau))\d x\d\tau&=\int_0^{\widetilde{T}}\int_\Omega\nabla Y(\tau):(I-(g_{i,j}(\tau)_{d\times d})\d\lambda^\tau(x)\d\tau\\
			&=\int_0^{\widetilde{T}}\int_\Omega\nabla Y(\tau):(I-\sum_{i=1}^d\alpha_i(\tau)v_i(\tau)\otimes v_i(\tau))\d\lambda^\tau(x)\d\tau\\
			&=\int_0^{\widetilde{T}}\int_\Omega\nabla Y(\tau):\sum_{i=1}^d c_i^\tau(x)(I-p_i^\tau(x)\otimes p_i^\tau(x))\d\lambda^\tau(x)\d\tau,
		\end{aligned}
	\end{equation}
	where we set
	\begin{align*}
		\bullet\,& c_i^t(x):=\alpha_i(t,x)+\frac{1}{d-1}\left(1-\sum_{j=1}^{d}\alpha_j(t,x)\right);\\
		\bullet\,& p_i^t(x):=v_i(t,x)/\sim_{\mc S^{d-1}}.
	\end{align*}
	Condition \eqref{eq:coefficients} now follows directly from \eqref{eq:propcoeff}, while \eqref{eq:density} is a byproduct of \eqref{eq:pointw}, the $L^1$-convergence $w^\eps\to 2\theta\chi_{Q^C}$ and the splitting $\d\lambda(t,x)=\d\lambda^t(x)\d t$.
	
	We now conclude by defining $V^t$ as in \eqref{eq:defVt} and $V$ as $\d V(t,x,p):=\d V^t(x,p)\d t$. In this way $(c)$ is automatically satisfied, moreover $(e)$ directly follows from \eqref{eq:imp} recalling the definition of first variation of a varifold \eqref{eq:firstvar}. Finally, the second equality in $(V)$ follows from \eqref{eq:tenslimit}.
	
	The fact that \eqref{eq:ass1} or \eqref{eq:ass2} implies $T_{\rm cr}=T$ is an immediate consequence of Lemma~\ref{lemma:boundphi}.

	\section{Sharp interface limit for Problem H}\label{sec:sharpinterfaceH}
	In this section we prove Theorem~\ref{thm:mainH}, so we tacitly assume all its hypotheses. 
	
	The argument follows the lines of previous section, with a crucial difference. In the energy balance \eqref{eq:EB} of problem \eqref{eq:problemP} all terms was nonnegative, so the first uniform bounds of Proposition~\ref{prop:unifboundP} were somehow for free. Instead, now the last term in the right-hand side of \eqref{eq:EBH} has no sign, and furthermore it depends on the chemical potential $\mu^\eps$, which is the most difficult term to bound.
	
		To overcome this issue, we exploit assumption \eqref{hyp:Htechnical}. This will allow to estimate the right-hand side of \eqref{eq:EBH} in terms of the gradient of $\sigma^\eps$ and of the energy $E^\eps$; we will then employ Gr\"onwall's inequality in order to bound the left-hand side of the energy balance \eqref{eq:EBH}.
		
		To this aim a great help is given by the a-priori $L^\infty$-bound of the nutrient $\sigma^\eps$.
		
		\subsection{Uniform bounds}
		We begin by computing the a-priori estimate for the integrand in the right-hand side of \eqref{eq:EBH}, employing assumption \eqref{hyp:Htechnical}.
		\begin{lemma}\label{lemma:technical}
			Assume \eqref{hyp:Htechnical}. Then for almost every $t\in (0,T)$ there holds
			\begin{equation*}
				(\mu^\eps(t),(\sigma^\eps(t)-1)H(\varphi^\eps(t)))_2\le \frac{\eps}{2}\|\nabla\sigma^\eps(t)\|_2^2+C E^\eps(t).
			\end{equation*}
		\end{lemma}
		\begin{proof}
			From the very definition of chemical potential we have
			\begin{align*}
				&(\mu^\eps(t),(\sigma^\eps(t)-1)H(\varphi^\eps(t)))_2\\
				=&\underbrace{-\eps\int_\Omega \Delta\varphi^\eps(t)(\sigma^\eps(t)-1)H(\varphi^\eps(t))\d x}_{=:I_1^\eps(t)} +\underbrace{\frac{1}{\eps}\int_\Omega F'(\varphi^\eps(t))(\sigma^\eps(t)-1)H(\varphi^\eps(t)) \d x}_{=:I_2^\eps(t)}.
			\end{align*}
			After integration by parts, since $H$ and $\sigma^\eps$ are valued in $[0,1]$ and recalling that $H$ is Lipschitz continuous, we can estimate
			\begin{align*}
				I_1^\eps(t)&=\eps \int_\Omega \nabla\varphi^\eps(t) \cdot( H(\varphi^\eps(t))\nabla \sigma^\eps(t)+(\sigma^\eps(t)-1)H'(\varphi^\eps(t))\nabla \varphi^\eps(t))\d x \\
				&\le \eps\left(\int_\Omega |\nabla\varphi^\eps(t)||\nabla \sigma^\eps(t)| \d x+ C\|\nabla\varphi^\eps(t)\|_2^2\right)\\
				&\le \frac{\eps}{2}\|\nabla \sigma^\eps(t)\|_2^2+ C\eps \|\nabla \varphi^\eps(t)\|_2^2\le \frac{\eps}{2}\|\nabla \sigma^\eps(t)\|_2^2+CE^\eps(t).
			\end{align*}
			As regards $I_2^\eps(t)$, we first use the fact that $(\sigma^\eps(t)-1) H(\varphi^\eps(t))$ is nonpositive and then we exploit \eqref{hyp:Htechnical} obtaining
			\begin{align*}
				I_2^\eps(t)&\le \frac{1}{\eps}\int_{\{F'(\varphi^\eps(t,\cdot))<0\}} |F'(\varphi^\eps(t))|(1-\sigma^\eps(t))H(\varphi^\eps(t)) \d x \\
				&\le \frac{C}{\eps}\int_{\{F'(\varphi^\eps(t,\cdot))<0\}} F(\varphi^\eps(t))\d x\le C E^\eps(t).
			\end{align*}
			Adding the two previous inequalities we conclude the proof.
		\end{proof}

		By means of Gr\"onwall's inequality we now infer a uniform bound for the left-hand side of the energy balance \eqref{eq:EBH}.
		\begin{lemma}\label{lemma:enboundH}
			There exists a constant $C>0$ for which the following uniform bound holds true:
			\begin{equation*}
				E^\eps(t)+\frac 12 \|\sigma^\eps(t)\|^2_{2}+\int_0^t\left(\|\nabla\mu^\eps(\tau)\|^2_{2}+\|\nabla\sigma^\eps(\tau)\|^2_{2}\right)\d\tau\le C,\qquad\text{for all }t\in [0,T].
			\end{equation*}
		\end{lemma}
		\begin{proof}
			We denote by $\mc E^\eps(t)$ the left-hand side of the energy balance \eqref{eq:EBH}. By Lemma \ref{lemma:technical}, and recalling assumption \eqref{eq:inenergybound}, we deduce
			\begin{align*}
				\mc E^\eps(t)\le \mc E_0+\int_0^t \left(\frac{\eps}{2}\|\nabla\sigma^\eps(\tau)\|_2^2+C E^\eps(\tau)\right)\d\tau.
			\end{align*}
			By absorbing the term with $\|\nabla\sigma^\eps\|_2^2$ in the left-hand side (choosing for instance $\eps$ smaller than 1) we finally infer
			\begin{equation*}
				\mc E^\eps(t)\le C\left(1+\int_0^t\mc E^\eps(\tau)\d\tau\right),
			\end{equation*}
			and we conclude by Gr\"onwall's inequality.
		\end{proof}
	
	As an immediate corollary of the previous two lemmas we obtain:
	\begin{cor}\label{cor:unifboundH}
		The following uniform estimates hold true:
		\begin{itemize}
			\item[(i)] $\sup\limits_{t\in [0,T]}E^\eps(t)\le C$;
			\item[(ii)] $\sup\limits_{t\in [0,T]}\||\varphi^\eps(t)|-1\|_{2}\le C\sqrt{\eps}$;
			\item[(iii)] $\displaystyle\|\varphi^\eps\|_{B([0,T];L^p(\Omega))}\le C$;
			\item[(iv)] $\|\nabla\sigma^\eps\|_{L^2(0,T;L^2(\Omega))}\le C$;
			\item[(v)] $\|\nabla\mu^\eps\|_{L^2(0,T;L^2(\Omega))}\le C$.
		\end{itemize}
	\end{cor}
	
	As before, in order to gain compactness for the phase variable $\varphi^\eps$, it is useful to consider the auxiliary function $w^\eps$ defined in \eqref{eq:w}.
	
	\begin{lemma}
		There hold:
		\begin{align*}
			&\bullet\, \|\nabla w^\eps(t)\|_{1}\le E^\eps(t)\le C,\quad\text{ for all }t\in [0,T];\\
			&\bullet \, \|w^\eps\|_{B([0,T];W^{1,1}(\Omega))}+\|w^\eps\|_{C^\frac{1}{16}([0,T];L^1(\Omega))}+\|\varphi^\eps\|_{C^\frac{1}{16}([0,T];L^2(\Omega))}\le C.
		\end{align*}
	\end{lemma}
	\begin{proof}
		The only difference with respect to the proof of Lemma~\ref{lemma:boundw} is the bound of $J^\eps_\eta$ in \eqref{eq:boundJ}, since it uses the equation for $\varphi^\eps$. However, the current situation is simpler, indeed exploiting again the fact that $H$ and $\sigma^\eps$ are valued in $[0,1]$ we can estimate using previous corollary:
		\begin{align*}
			J^\eps_\eta(t,s)=&\int_s^t(\dot\varphi^\eps(\tau),\varphi^\eps_\eta(t)-\varphi^\eps_\eta(s))_2\d \tau\\
			=&{-}\int_{s}^{t}(\nabla\mu^\eps(\tau),\nabla\varphi^\eps_\eta(t){-}\nabla\varphi^\eps_\eta(s))_2\d\tau+\int_s^t((\sigma^\eps(\tau)-1)H(\varphi^\eps(\tau)),\varphi^\eps_\eta(t)-\varphi^\eps_\eta(s))_2\d\tau\\
			\le & \|\nabla\mu^\eps\|_{L^2(0,T;L^2(\Omega))}(\|\nabla\varphi^\eps_\eta(t)\|_2+\|\nabla\varphi^\eps_\eta(s)\|_2)(t-s)^\frac 12+\left(\|\varphi^\eps_\eta(t)\|_1+\|\varphi^\eps_\eta(s)\|_1\right)(t-s)\\
			\le & \frac C\eta(t-s)^\frac 12+ C(t-s)\le C\left(1+\frac 1\eta\right)(t-s)^\frac 12,
		\end{align*}
		where we used \eqref{eq:etaa} and \eqref{eq:etab}. One then concludes by arguing exactly as in Lemma~\ref{lemma:boundw}.
	\end{proof}
	
	Also in this section we bound the derivative of the nutrient in order to deduce strong compactness, which will be used later on.
	
	\begin{lemma}\label{lemma:sigmadotH}
		The derivative $\dot\sigma^\eps$ is bounded in $L^2(0,T;H^1(\Omega)^*)$ uniformly with respect to $\eps$.
	\end{lemma}
	\begin{proof}
		By using the equation \eqref{eq:weakeqH}$_2$, and exploiting again the uniform bounds of $H$ and $\sigma^\eps$, for all $\psi\in H^1(\Omega)$ we infer
		\begin{equation*}
			|\langle\dot\sigma^\eps(t),\Psi\rangle_{H^{1}(\Omega)}|\le \|\nabla\sigma^\eps(t)\|_2\|\nabla \Psi\|_2+\|\Psi\|_{1}\le C(1+\|\nabla\sigma^\eps(t)\|_2)\|\Psi\|_{H^1(\Omega)},
		\end{equation*}
		and we conclude by $(iv)$ in Corollary~\ref{cor:unifboundH}.
	\end{proof}
	
	We now state, in a slightly different form useful for our purposes, the a-priori estimate for the chemical potential $\mu^\eps$.
	\begin{lemma}\label{lemma:boundmuH}
		If $T<1-c_0$, where $c_0$ is given by \eqref{eq:inmassbound}, then there exist $\eps_0\in (0,1)$ and a constant $C>0$, which blows up as $T$ goes to $1-c_0$, for which
		\begin{equation}\label{eq:boundmuH}
			|[\mu^\eps(t)]|\le C(E^\eps(t)+\|\nabla \mu^\eps(t)\|_2),\quad\text{for almost every }t\in (0,T)\text{ and for every }\eps\in (0,\eps_0).
		\end{equation}
		In particular, $\mu^\eps$ is uniformly bounded in $L^2(0,T;H^1(\Omega))$.
	\end{lemma} 
	\begin{proof}
		Inequality \eqref{eq:boundmuH} is a byproduct of Lemma~\ref{lemma:boundmu} once we check the validity of \eqref{eq:boundmass}. This easily follows from Remark~\ref{rmk:Lip}, indeed by \eqref{eq:inmassbound} we have
		\begin{equation*}
			|[\varphi^\eps(t)]|\le |[\varphi^\eps_0]|+t\le c_0+T=:m_0<1.
		\end{equation*}
		Moreover, when $T$ goes to $1-c_0$ one has that $m_0$ goes to $1$, hence by Lemma~\ref{lemma:boundmu} one obtains that the constant $C$ blows up.
	\end{proof}

	\subsection{Compactness}
	Similarly to Proposition~\ref{prop:compactnessP}, we now deduce the following result.
		\begin{prop}\label{prop:compactnessH}
			There exist a (non relabelled) subsequence $\eps\to 0$, a measurable set $Q^C\subseteq Q$, a function $\sigma\in L^2(0,T;H^1(\Omega))$ with $0\le \sigma\le 1$ almost everywhere in $Q$, and Radon measures $\lambda, \lambda_{i,j}$ on $\overline{Q}$ for $i,j=1,\dots,d$ such that the following hold true:
			\begin{enumerate}
				\item $w^\eps\xrightarrow[\eps\to 0]{}2\theta \chi_{Q^C}$ in $C([0,T];L^1(\Omega))$, and $\chi_{Q^C}$ is in $C^\frac 18([0,T];L^1(\Omega))\cap B([0,T];BV(\Omega))$;
				\item $\varphi^\eps\xrightarrow[\eps\to 0]{}-1+2\chi_{Q^C}=:\varphi$ in $C([0,T];L^2(\Omega))$, and $\varphi^\eps(t)\xrightarrow[\eps\to 0]{}\varphi(t)$ in $L^q(\Omega)$ for all $q\in [1,p)$ and $t\in [0,T]$;
				\item $\sigma^\eps \xrightharpoonup[\eps\to 0]{}\sigma$ in $L^2(0,T;H^1(\Omega))$ and $\sigma^\eps \xrightarrow[\eps\to 0]{}\sigma$ in $L^q(Q)$ for all $q\ge 1$;
				\item $e^\eps(\varphi^\eps)\d t\d x\xrightharpoonup[\eps\to 0]{\ast}\lambda$ and $\eps\partial_i\varphi^\eps\partial_j\varphi^\eps\d t\d x\xrightharpoonup[\eps\to 0]{\ast}\lambda_{i,j}$ in the sense of measures in $\overline Q$.
			\end{enumerate}
			If in particular $T<1-c_0$, where $c_0$ is given by \eqref{eq:inmassbound}, then there also exist a bounded measurable function $E\colon [0,T]\to [0,+\infty)$ and a function $\mu\in L^2(0,T;H^1(\Omega))$ such that
			\begin{enumerate}[start=5]        
				\item $\mu^\eps \xrightharpoonup[\eps\to 0]{}\mu$ in $L^2(0,T;H^1(\Omega))$;
				\item $E^\eps(t)\xrightarrow[\eps\to 0]{}E(t)$ for almost every $t\in (0,T)$.
			\end{enumerate}
		\end{prop}
	\begin{proof}
		The validity of $(1)-(5)$ follows by arguing as in Proposition~\ref{prop:compactnessP}, the only difference being the strong convergence in $L^q(Q)$ of the nutrient variable which can be obtained by Aubin-Lions theorem (exploiting Lemma~\ref{lemma:sigmadotH}) combined with the uniform bound $0\le \sigma^\eps\le 1$.
		
		As regards $(6)$ we first notice that by the energy balance \eqref{eq:EBH} the functions
		\begin{equation*}
			G^\eps(t):=E^\eps(t)+\frac 12 \|\sigma^\eps(t)\|_2^2-\int_0^t(\mu^\eps(\tau),(\sigma^\eps(\tau)-1)H(\varphi^\eps(\tau)))_2\d\tau,
		\end{equation*}
		are nondecreasing and by Lemmas~\ref{lemma:enboundH} and \ref{lemma:boundmuH} they also are uniformly bounded. Hence, by Helly's Selection Theorem they pointwise converge to a function $G$ in $[0,T]$. On the other hand $\sigma^\eps(t)\xrightarrow[\eps\to 0]{}{\sigma(t)}$ in $L^2(\Omega)$ for almost every $t\in (0,T)$, and since by Dominated Convergence Theorem there holds $(\sigma^\eps-1)H(\varphi^\eps)\xrightarrow[\eps\to 0]{}(\sigma-1)H(\varphi)$ in $L^2(Q)$ we also have
		\begin{equation*}
			\lim\limits_{\eps\to 0}\int_0^t(\mu^\eps(\tau),(\sigma^\eps(\tau)-1)H(\varphi^\eps(\tau)))_2\d\tau=\int_0^t(\mu(\tau),(\sigma(\tau)-1)H(\varphi(\tau)))_2\d\tau.
		\end{equation*}
		For almost every $t\in (0,T)$ this finally yields
		\begin{equation*}
			\lim\limits_{\eps\to 0}E^\eps(t)=G(t)-\frac 12 \|\sigma(t)\|_2^2+\int_0^t(\mu(\tau),(\sigma(\tau)-1)H(\varphi(\tau)))_2\d\tau=:E(t),
		\end{equation*} 
		and we conclude.
	\end{proof}
	
	\subsection{Proof of Theorem~\ref{thm:mainH}}
		We finally have all the ingredients for proving Theorem~\ref{thm:mainH}.
		
		Note that $(I)-(III)$ have been already obtained in Proposition~\ref{prop:compactnessH}. Let us now fix $\widetilde{T}<1-c_0$, so that $(IV)$ as well as $(a), (b)$ in Definition~\ref{def:varifoldsol} are fulfilled. In order to deduce $(d)$ we let $\eps\to 0$ in \eqref{eq:weakeqH}, which after integration by parts in time reads as:
		
		\begin{align*}
			\int_0^{\widetilde T}\bigg({-}(1+\varphi^\eps(\tau),\dot\Phi(\tau))_2&+(\nabla\mu^\eps(\tau),\nabla\Phi(\tau))_{2}-((\sigma^\eps(\tau){-}1)H(\varphi^\eps(\tau)),\Phi(\tau))_{2}\bigg)\d\tau\\
			&=(1+\varphi^\eps_0,\Phi(0))_2;
		\end{align*}
		\begin{align*}
			\int_0^{\widetilde T}\bigg({-}(\sigma^\eps(\tau),\dot\Psi(\tau))_{2}&+(\nabla\sigma^\eps(\tau),\nabla\Psi(\tau))_{2}+(\sigma^\eps(\tau)H(\varphi^\eps(\tau)),\Psi(\tau))_{2}\bigg)\d\tau=(\sigma^\eps_0,\Psi(0))_{2}.
		\end{align*}
		Notice indeed that the two terms involving $H$ strongly converge in $L^2((0,\widetilde{T})\times\Omega)$ due to Dominated Convergence Theorem.
		
		The construction of the varifold $V$ is done exactly as in the proof of Theorem~\ref{thm:main}, yielding $(V)$ up to checking the validity of $(f)$ in Definition~\ref{def:varifoldsol}. To this aim we exploit weak-lower semicontinuity and the energy balance \eqref{eq:EBH} to infer for almost every $0\le s\le t\le \widetilde{T}$ the following inequality
		\begin{align*}
			&\lambda^t(\overline{\Omega})+\frac 12 \|\sigma(t)\|^2_{2}+\int_s^t\left(\|\nabla\mu(\tau)\|^2_{2}+\|\nabla\sigma(\tau)\|^2_{2}+\|\sqrt{H(\varphi(\tau))}\sigma(\tau)\|^2_{2}\right)\d\tau\\
			\le&\liminf\limits_{\eps\to 0}\left(E^\eps(t){+}\frac 12 \|\sigma^\eps(t)\|^2_{2}{+}\int_s^t\left(\|\nabla\mu^\eps(\tau)\|^2_{2}{+}\|\nabla\sigma^\eps(\tau)\|^2_{2}{+}\|\sqrt{H(\varphi^\eps(\tau))}\sigma^\eps(\tau)\|^2_{2}\right)\d\tau\right)\\	
			=&\lim\limits_{\eps\to 0}\left(E^\eps(s)+\frac 12 \|\sigma^\eps(s)\|^2_{2}+ \int_s^t(\mu^\eps(\tau),(\sigma^\eps(\tau)-1)H(\varphi^\eps(\tau)))_2\d\tau\right)\\
			=&\lambda^s(\overline{\Omega})+\frac 12 \|\sigma(s)\|^2_{2} +\int_s^t(\mu(\tau),(\sigma(\tau)-1)H(\varphi(\tau)))_2\d\tau.
		\end{align*}
		
		We thus conclude if we prove that the critical time $T_{\rm cr}$ can be characterized as in \eqref{eq:TmaxH}. We first observe that by Remark~\ref{rmk:Lip} and by the convergence $\varphi^\eps\xrightarrow[\eps\to 0]{}\varphi$ in $C([0,T];L^2(\Omega))$ one deduces that the mass $[\varphi^\eps(\cdot)]$ uniformly converges to $[\varphi(\cdot)]$ in $[0,T]$. As a consequence the map $t\mapsto [\varphi(\cdot)]$ is nondecreasing, indeed $[\varphi^\eps(\cdot)]$ was nondecreasing.
		
		Now two cases are possible: either $[\varphi(1-c_0)]=-1$, whence $T_{\rm cr}=1-c_0$ and we conclude, or $|[\varphi(1-c_0)]|<1$. In this second case, by uniform convergence of the mass we deduce that
		\begin{equation*}
			|[\varphi^\eps(1-c_0)]|\le c_1<1,
		\end{equation*}
		definitively, and so conditions \eqref{eq:hypinitial} are satisfied at time $1-c_0$.
		
		Hence Lemma \ref{lemma:boundmuH} still holds true till a second time $T_2>1-c_0$, where one of the previous two alternatives must happen. Iterating this procedure finally leads to \eqref{eq:TmaxH}, and so we conclude.
	
	\bigskip
	
	\noindent\textbf{Acknowledgements.}
	The research performed here has been supported by the MIUR-PRIN Grant 2020F3NCPX 
	``Mathematics for industry 4.0 (Math4I4)''. 
	
	The present paper also benefits from the support of 
	the GNAMPA (Gruppo Nazionale per l'Analisi Matematica, la Probabilit\`a e le loro Applicazioni)
	of INdAM (Istituto Nazionale di Alta Matematica).
	
	E.R. also acknowledges the support of Next Generation EU Project No.P2022Z7ZAJ (A unitary mathematical framework for modelling muscular dystrophies).

	{\small
		
		\vspace{15pt} (Filippo Riva) Universit\`{a} degli Studi di Pavia, Dipartimento di Matematica ``Felice Casorati'', \par
		\textsc{Via Ferrata, 5, 27100, Pavia, Italy}
		\par
		\textit{e-mail address}: \textsf{filippo.riva@unipv.it}
		\par
		\textit{Orcid}: \textsf{https://orcid.org/0000-0002-7855-1262}
		\par
		\vspace{15pt} (Elisabetta Rocca) Universit\`{a} degli Studi di Pavia, Dipartimento di Matematica ``Felice Casorati'' and\par IMATI-C.N.R.,\par
		\textsc{Via Ferrata, 5, 27100, Pavia, Italy}
		\par
		\textit{e-mail address}: \textsf{elisabetta.rocca@unipv.it}
		\par
		\textit{Orcid}: \textsf{https://orcid.org/0000-0002-9930-907X}

	}
	

\bigskip
	
	\begin{thebibliography}{1}
		
		\bibitem{AL14}
		{\sc H. Abels and D. Lengeler}, {\em On sharp interface limits for diffuse interface models for two-phase flows}, Interfaces and Free Boundaries, 16 (2014), pp.~395--418.
		
		\bibitem{BLM08}
		{\sc N. Bellomo, N.K. Li and P.K. Maini}, 
		{\em On the foundations of cancer modeling: selected topics, speculations, and perspectives}, Math. Models Methods Appl. Sci., 4 (2008), pp.~593--646.
		
		\bibitem{CavRocWu}
		{\sc C. Cavaterra, E. Rocca, and H. Wu}, {\em Long-time Dynamics and Optimal Control of a Diffuse Interface Model for Tumor Growth}, Applied Mathematics \& Optimization, 83 (2021), pp.~739--787.
		
		\bibitem{Chen}
		{\sc  X.~Chen}, {\em Global asymptotic limit of solutions of the Cahn-Hilliard equation}, J. Differential Geometry, 44 (1996), pp.~262--311. 
		
		\bibitem{CWSL14} 
		{\sc Y. Chen, S.M. Wise, V.B. Shenoy, and J.S. Lowengrub}, {\em A stable scheme for a nonlinear, multiphase tumor growth model with an elastic membrane}, Int. J. Numer. Methods Biomed. Eng., 30 (2014), pp.~726--754. 
		
		\bibitem{CGRS1}
		{\sc P. Colli, G. Gilardi, E. Rocca, and J. Sprekels}, {\em Vanishing viscosities and error estimate for a Cahn-Hilliard type phase field system related to tumor growth}, Nonlinear Anal. Real World Appl., 26 (2015), pp.~93--108.
		
		\bibitem{CGRS2} 
		{\sc P. Colli, G. Gilardi, E. Rocca, and J. Sprekels}, {\em Asymptotic analyses and error estimates for a Cahn-Hilliard type phase field system modelling tumor growth},  Discrete Contin. Dyn. Syst. Ser. S, 10 (2017), pp.~37--54
		
		\bibitem{CEtAl08}
		{\sc V. Cristini, H.B. Frieboes, X. Li, J.S. Lowengrub, P. Macklin, S. Sanga, S.M. Wise, and X. Zheng}, {\em Nonlinear modeling and simulation of tumor growth}, in: N. Bellomo, M. Chaplain, E. de Angelis (Eds.), Selected Topics in Cancer Modeling: Genesis, Evolution, Immune Competition, and Therapy, in: Modeling and Simulation in Science, Engineering and Technology, Birkhauser, 2008.
		
		\bibitem{CLW09} 
		{\sc V. Cristini, X. Li, J.S. Lowengrub, and S.M. Wise}, {\em Nonlinear simulations of solid tumor growth using a mixture model: invasion and branching}, J. Math. Biol., 58 (2009), pp.~723--763.
		
		\bibitem{CL10}
		{\sc V. Cristini and  J. Lowengrub}, {\em Multiscale modeling of cancer. An Integrated Experimental and Mathematical Modeling Approach}, Cambridge Univ. Press, 2010.
		
		\bibitem{DFRS15} 
		{\sc M. Dai, E. Feireisl, E. Rocca, G. Schimperna, and M. Schonbek}, {\em Analysis of a diffuse interface model of multispecies tumor growth}, Nonlinearity, 30 (2017), 1639.
		
			\bibitem{FHLS} 
			{\sc J. Fischer, S. Hensel, T. Laux, and T. M. Simon}, {\em The local structure of the energy landscape in multiphase mean curvature flow: weak-strong uniqueness and stability of evolutions}, to appear in Journal of the European Mathematical Society.
			\bibitem{FLS} 
			{\sc J. Fischer, T. Laux, and T. M Simon}, {\em Convergence rates of the Allen--Cahn equation to Mean Curvature Flow: a short proof based on relative entropies}, SIAM Journal on Mathematical Analysis, 52 (2020), pp.~6222--6233. 
		
		\bibitem{FBG06} 
		{\sc A. Fasano, A. Bertuzzi, and A. Gandolfi}, {\em Mathematical modelling of tumour growth and treatment, in: Complex Systems in Biomedicine}, in: Biomedical and Life Science, Springer, 2006.
		
		
		\bibitem{FBM07} 
		{\sc A. Friedman, N. Bellomo, and P.K. Maini}, {\em Mathematical analysis and challenges arising from models of tumor growth}, Math. Models Methods Appl. Sci., 17 (2007), pp.~1751--1772. 
		
		\bibitem{FriGrasRoc}
		{\sc  S.~Frigeri, M.~Grasselli, and E.~Rocca}, {\em On a diffuse interface model of tumor growth}, European J. Appl. Math., 26 (2015), pp.~215--243. 
		
		\bibitem{FLRS}
		{\sc S. Frigeri, K.-F. Lam, E. Rocca, and G. Schimperna}, {\em On a multi-species Cahn-Hilliard-Darcy tumor growth model with singular potentials}, Comm Math Sci., 16 (2018), pp.~821--856.
		
		\bibitem{GarLam}
		{\sc  H.~Garcke, and K.~F.~Lam}, {\em Well-posedness of a Cahn-Hilliard system modelling tumour growth with chemotaxis and active transport}, European J. Appl. Math., 28 (2017), pp.~284--316. 
		
		\bibitem{GarLamRoc}
		{\sc  H.~Garcke, K.~F.~Lam, and E.~Rocca}, {\em Optimal control of treatment time in a diffuse interface model of tumor growth}, Appl. Math. Optim., 78 (2018), pp.~495--544. 
		
		\bibitem{GarLamSitSty}
		{\sc  H.~Garcke, K.~F.~Lam, E.~Sitka, and V.~Styles}, {\em A Cahn–Hilliard–Darcy model for tumour growth with chemotaxis and active transport}, Math. Models Methods Appl. Sci., 26 (2016), pp.~1095--1148. 
		
			\bibitem{GilMirSchimp} 
			{\sc G. Gilardi, A. Miranville, and G. Schimperna}, 
			{\em Long time behavior of the Cahn-Hilliard equation with irregular potentials and dynamic boundary conditions}, 
			Chin. Ann. Math., 31B (2010), pp.~679--712.
		
		\bibitem{GRS} 
		{\sc A. Giorgini, K.-F. Lam, E. Rocca, and G. Schimperna}, 
		{\em On the Existence of Strong Solutions to the Cahn-Hilliard-Darcy system with mass source}, 
		SIAM J. Math. Anal., 54 (2022), pp.~737--767.
		
		\bibitem{GP07} 
		{\sc L. Graziano and L. Preziosi}, {\em Mechanics in tumor growth}, in: F. Mollica, L. Preziosi, K.R. Rajagopal (Eds.), Modeling of Biological Materials, in: Modeling and Simulation in Science, Engineering and Technology, Birkhauser, 2007.
		
		\bibitem{HZO11}
		{\sc A. Hawkins-Daruud, K. G. van der Zee, and J. T. Oden}, {\em Numerical simulation of a thermodynamically consistent four-species tumor growth model}, Int. J. Numer. Math. Biomed. Engng., 28 (2011), pp.~3--24. 
		
			\bibitem{HensLauxentropy}
			{\sc  S.~Hensel, and T.~Laux}, {\em A new varifold solution concept for mean curvature flow: convergence of the Allen-Cahn equation and weak-strong uniqueness}, to appear in J. Differential Geometry.
			\bibitem{HensLaux}
			{\sc  S.~Hensel, and T.~Laux}, {\em BV solutions for mean curvature flow with constant contact angle: Allen-Cahn approximation and weak-strong uniqueness}, Indiana Univ. Math. J., 73 (2024), pp.~111--148. 
			\bibitem{HensStin}
			{\sc  S.~Hensel, and K.~Stinson}, {\em Weak solutions of Mullins–Sekerka flow as a Hilbert space gradient flow}, Arch. Rational Mech. Anal., 248 (2024), n.~8. 
		
		\bibitem{Hil}
		{\sc D. Hilhorst, J. Kampmann, T. N. Nguyen, and K. G. van der Zee}, {\em  Formal asymptotic limit of a diffuse-interface tumor-growth model}, Math. Models Methods Appl. Sci., 25 (2015), pp.~1011--1043.		
		
		
		\bibitem{JWZ15} 
		{\sc J. Jiang, H. Wu, and S. Zheng}, {\em Well-posedness and long-time behavior of a non-autonomous Cahn-Hilliard-Darcy system with mass source modeling tumor growth}, J. Differential Equations, 259 (2015), pp.~3032--3077.

			\bibitem{KroemLaux}
			{\sc M.~Kroemer, and T.~Laux}, {\em The Hele–Shaw flow as the sharp interface limit of the Cahn–Hilliard equation with disparate mobilities}, Communications in Partial Differential Equations, 47 (2022), pp.~2444--2486. 
			\bibitem{Laux}
			{\sc  T.~Laux}, {\em Weak-strong uniqueness for volume-preserving
				mean curvature flow}, Rev. Mat. Iberoam., 40 (2024), pp.~93--110. 
			\bibitem{LauxSim}
			{\sc  T.~Laux, and T.~M.~Simon}, {\em Convergence of the Allen-Cahn equation to multiphase mean curvature flow}, Communications on Pure and Applied Mathematics, 71 (2018), pp.~1597--1647. 
			\bibitem{LauxStinson}
			{\sc  T.~Laux, and K.~Stinson}, {\em Sharp interface limit of the Cahn-Hilliard reaction model for lithium-ion batteries}, Mathematical Models and Methods in Applied Sciences, 33 (2023), pp.~2557--2585. 
		
		\bibitem{Le}
		{\sc  N.~Q.~Le}, {\em A Gamma-convergence approach to the Cahn-Hilliard equation}, Calc. Var. Partial Differ. Equ., 32 (2008), pp.~499--522. 
		
		\bibitem{LEtAl10}
		{\sc J. S. Lowengrub, H. B. Frieboes, F. Lin, Y.-L. Chuang, X. Li, P. Macklin, S. M. Wise, and V. Cristini}, {\em Nonlinear modelling of cancer: bridging the gap between cells and tumors}, Nonlinearity, 23 (2010), R1--R91.		
		
		\bibitem{KS}  
		{\sc P. Knopf and A. Signori}, {\em Existence of weak solutions to multiphase Cahn-Hilliard-Darcy and Cahn-Hilliard-Brinkman models for stratified tumor growth with chemotaxis and general source terms}, Comm. Partial Differential Equations, 47 (2022), pp.~233–278.
		
		\bibitem{MelRoc}
		{\sc  S.~Melchionna, and E.~Rocca}, {\em Varifold solutions of a sharp interface limit of a diffuse interface model for tumor growth}, Int. Free Boundaries, 19 (2017), pp.~571--590. 
		
		\bibitem{MRS} 
		{\sc A. Miranville, E. Rocca, and G. Schimperna},  {\em On the long time behavior of a tumor growth model}, Journal of Differential Equations, 67 (2019), pp.~261--2642.
		
		\bibitem{MirZelik} 
			{\sc A. Miranville, and S. Zelik},  {\em The Cahn-Hilliard equation with singular potentials and dynamic boundary conditions}, Discrete and Continuous Dynamical Systems, 28 (2010), pp.~275--310.
		
		\bibitem{RocSca}
		{\sc  E.~Rocca, and R.~Scala}, {\em A rigorous sharp interface limit of a diffuse interface model related to tumor growth}, J. Nonlinear Sci., 27 (2017), pp.~847--872. 
		
		\bibitem{RocSchim}
		{\sc  E.~Rocca, and G.~Schimperna}, {\em Universal attractor for some singular phase transition systems}, Physica D, 192 (2004), pp.~279--307. 
		
		\bibitem{S11} 
		{\sc S. Serfaty}, {\em Gamma-Convergence of gradient flows on Hilbert and metric spaces and applications},  Discrete Contin. Dyn. Syst, 31 (2011), pp.~1427--1451. 
		
		\bibitem{T05} 
		{\sc Y. Tonegawa}, {\em A diffuse interface whose chemical potential lies in Sobolev spaces}, Ann. Scuola Norm. Sup. Pisa Cl. Sci., 4 (2005), pp.~487--510. 
		
		\bibitem{WLFC08} 
		{\sc S.M. Wise, J.S. Lowengrub, H.B. Frieboes, and V. Cristini}, {\em Three-dimensional multispecies nonlinear tumor growth—I: model and numerical method}, J. Theoret. Biol., 253 (2008), pp.~524--543.
		
		
		
	\end{thebibliography}
\end{document}